\documentclass[11pt,reqno]{amsart}
\usepackage{amsmath,amsthm,amsfonts,amssymb,amscd,amsgen,latexsym,euscript,graphicx,overpic,color,mathrsfs}
\usepackage[normalem]{ulem}

\def\bR{\mathbb{R}}
\def\cD{\mathcal{D}}
\def\cM{\mathcal{M}}
\def\cR{\mathcal{R}}
\def\cH{\mathcal{H}}

\def\longrightarrow{\to}

\def\cs{{\rm cs}}
\def\cu{{\rm cu}}
\def\s{{\rm s}}
\def\u{{\rm u}}
\def\ss{{\rm ss}}
\def\uu{{\rm uu}}

\newtheorem{thmy}{Theorem}

\newtheorem{theo}{Theorem}[section]
\newtheorem{theorem}[theo]{Theorem}
\newtheorem{corollary}[theo]{Corollary}
\newtheorem{proposition}[theo]{Proposition}
\newtheorem{lemma}[theo]{Lemma}
\newtheorem*{claim}{Claim}
\theoremstyle{definition}

\newtheorem{remark}[theo]{Remark}
\newtheorem*{remark*}{Remark}
\newtheorem{definition}[theo]{Definition}
  
\newcommand{\eqdef}{:=}

\DeclareMathOperator{\Per}{Per}
\DeclareMathOperator{\card}{\#}
\DeclareMathOperator{\diam}{diam}

\DeclareMathOperator{\cl}{cl}

\title[Geodesic flows modeled by expansive flows]{Geodesic flows modeled by expansive flows}


\thanks{
This paper was partially supported by CNPq (Brazil).
}

\author[K. Gelfert]{Katrin Gelfert}
\address{Instituto de Matem\'atica, Universidade Federal do Rio de Janeiro, Cidade Universit\'aria - Ilha do Fund\~ao, Rio de Janeiro 21945-909,  Brazil}
\email{gelfert@im.ufrj.br}

\author[R.~O.~Ruggiero]{Rafael O.~Ruggiero}
\address{Departamento de Matem\'atica PUC-Rio, Rua Marqu\'es de S\~ao Vicente 225, Rio de Janeiro 22543-900, Brazil}
\email{rorr@mat.puc-rio.br}

\begin{document}

\begin{abstract}
Given a smooth compact surface without focal points and of higher genus, it is shown that its geodesic flow is semi-conjugate to a continuous expansive flow with a local product structure such that  the semi-conjugation preserves time-parametrization. It is concluded that the geodesic flow has a unique measure of maximal entropy.
\end{abstract}

\maketitle

\tableofcontents

\section{Introduction}

The resemblance between the geodesic flow of a compact surface of genus greater than one and without conjugate points and the geodesic flow of a compact surface with constant negative
curvature has been source of deep research in geometry and topological dynamics from the beginning of the 20th century. The first and perhaps most influential work in the subject
is the  work by Morse \cite{kn:Morse} showing that every globally minimizing geodesic in the universal covering of a compact surface of genus greater than one is ``shadowed'' by a geodesic in the hyperbolic plane. This beautiful result strongly suggested that the geodesic flow of a compact surface without conjugate points and of higher genus should be semi-conjugate to
the geodesic flows of a constant negative curvature surface.

Recall that two continuous flows $\phi_{t}\colon Y \longrightarrow Y$ and $\psi_{t}\colon X \longrightarrow X$ acting on compact metric spaces $Y$ and $X$ are \emph{semi-conjugate} if there exists a continuous and surjective map $\chi\colon Y\longrightarrow X$ such that for each $p \in Y$ there exists a
continuous and surjective reparametrization $\rho_{p}\colon {\mathbb R} \longrightarrow {\mathbb R}$ such that
$$
	(\chi\circ\phi_{t})(p) = (\psi_{\rho_{p}(t)}\circ\chi)(p).
$$
The map $\chi$ is called a \emph{semi-conjugacy}; if $\chi$ is a homeomorphism then it is called a \emph{conjugacy} and the flows are said to be \emph{semi-conjugate}.  Observe that $\chi$ maps orbits  to orbits.
We say that $\phi_t$ is \emph{time-preserving semi-conjugate} to $\psi_t$ if $\rho_p$ is the identity for every $p$.
The existence of conjugacies to nearby flows in a $C^{1}$ neighborhood characterizes \emph{structurally stable
flows}, that is, Axiom A flows and, in particular, Anosov flows. Structural stability theory was developed in the 60s and 70s and its main ideas  paved the path
to study also systems which show only weaker forms of stability like topological stability. 
A system is \emph{$C^{k}$-topologically stable} if it is semi-conjugate to any nearby $C^{k}$ system.
Of course, $C^{k}$-structurally stable systems are $C^{k}$-topologically stable. But the converse is not true and there are many well known counterexamples, many of them in the category
of expansive, non-hyperbolic systems.

\begin{definition}{\rm
A continuous flow $\psi_{t}\colon X \longrightarrow  X$ without singular points on a  metric space $(X,d)$ is \emph{expansive}%
\footnote{Observe that, in fact, our definition of expansivity made above is slightly stronger than in~\cite{kn:BW72}, however it appears naturally in the context of expansive geodesic flows (see, for example,~\cite{kn:Rug94,kn:Ruggiero97}). In the context of geodesic flows on compact manifolds without conjugate points both definitions are equivalent.} 
if there exists $\varepsilon>0$ such that for every $x \in X$ and 
for every $y \in X$ for which there exists a continuous surjective function $\rho\colon \bR \longrightarrow \bR$ with $\rho(0)=0$ satisfying
\[
 	d(\psi_{t}(x), \psi_{\rho(t)}(y)) \leq \varepsilon
\]	
for every $t \in {\mathbb R}$ we have $\psi_{t(y)}(x)=y$ for some $\lvert t(y) \rvert <\varepsilon$. We call such $\varepsilon$ an \emph{expansivity constant}.
}\end{definition}

The definition of expansive homeomorphism introduced by Bowen is previous to the above definition for flows (see also~\cite{kn:BW72}), and Bowen's
study of expansive homeomorphisms showed how to find weak stability properties in non-hyperbolic geodesic flows
using a dynamical approach rather than the global geometry approach suggested by Morse's ideas.

Bowen \cite{Bow:73c} and Bowen and Walters \cite{kn:BW72} pointed out that expansive systems with a local product structure are $C^{0}$-topologically stable (see also \cite{Tho:91,Fra:77} in the case of flows);
Lewowicz \cite{kn:Lewowicz} and Hiraide \cite{kn:Hiraide} showed that expansive homeomorphisms of compact surfaces have an ``almost" local product structure.
Paternain \cite{kn:MPaternain} and Inaba and Matsumoto \cite{kn:IM} extended Lewowicz's work to show that expansive geodesic flows of compact surfaces have a local
product structure and hence that they are $C^0$-topologically stable. Ruggiero \cite{kn:Ruggiero97} showed that expansive geodesic flows of compact
manifolds without conjugate points have a local product structure, which also implies topological stability.

The global geometry point of view of weak stability theory of geodesic flows enjoyed an large development in the 70s. One can mention the works of Eberlein \cite{kn:Eberlein} and Eberlein and O'Neill \cite{kn:EO} about visibility manifolds extending Morse's~\cite{kn:Morse} work: quasi-geodesics in visibility manifolds are close to true geodesics. The theory of visibility manifolds went further and introduced a whole body of tools to study coarse hyperbolic geometry of manifolds.
Thurston and Gromov \cite{kn:Gromov} introduced the notion of  hyperbolic groups and not only extended Eberlein's theory of visibility manifolds but created a rich theory to study coarse hyperbolic geometry in very general metric spaces. All the above results suggest that the family of geodesic flows which are semi-conjugate to a geodesic flow of a manifold with negative curvature could be much larger than only the family of geodesic flows of compact surfaces without conjugate points.

In the 80s,  Ghys \cite{kn:Ghys} proved the existence of a semi-conjugacy between the
geodesic flow of a compact surface without conjugate points and of genus greater than one and a hyperbolic \emph{geodesic} flow. 
However, in general, this semi-conjugacy is \emph{not} time-preserving, that is, the reparametrization $\rho_{p}$ is not the identity. While any semi-conjugacy between discrete systems is time preserving by definition, the existence of
such a semi-conjugacy for flows is a much more delicate issue. Time-preserving semi-conjugacies are intimately related to
spectral rigidity problems in continuous dynamical systems and in particular in Riemannian geometry: two geodesic flows which are time-preserving semi-conjugate have the same marked length spectrum of periodic geodesics. Indeed, the work of Otal \cite{kn:Otal}, Croke \cite{kn:Croke}, and Croke and Fathi \cite{kn:CF}, shows that if this is the case of the geodesic flows of two compact surfaces without
conjugate points and of genus greater than one then both surfaces are in fact isometric. 

In view of the above, the main result result of the present article looks perhaps surprising. 

\begin{thmy} \label{quotient}
Let $(M,g)$ be a $C^{\infty}$ compact connected boundaryless surface without focal points and genus greater than one. Let $\phi_{t}\colon T_{1}M \longrightarrow T_{1}M$ be its geodesic flow.

Then there exists a compact 3-manifold $X$ diffeomorphic to $T_{1}M$ and a continuous  flow
$\psi_{t} \colon X \longrightarrow X$ which is expansive and has a local product structure such that $\phi_t$ is time-preserving semi-conjugate to $\psi_t$.
\end{thmy}

So expansive flows arise as models of geodesic flows of compact surfaces without focal points and of genus greater than one up to
time-preserving semi-conjugacies. 
By the above mentioned spectral rigidity, we cannot expect that the  the model flow $\psi_{t}$ is the geodesic flow of a compact surface if the original flow is non-expansive. This is why Ghys's method \cite{kn:Ghys}  does not provide a \emph{time-preserving} semi-conjugacy.  Indeed,
if the model flow $\psi_{t}$ was a geodesic flow then we would know that the corresponding surface would have no conjugate points by the result
of Paternain \cite{kn:MPaternain}. Therefore, we could apply the spectral rigidity results mentioned above to show that both surfaces would have to be isometric.
But the initial geodesic flow will not be expansive in general due to flat strips  in the unit tangent bundle (see Section~\ref{sec:2}).

The existence of a time-preserving semi-conjugacy to an expansive flow has some interesting applications in topological dynamics.
In fact, this expansive flow inherits most of the features of the topological hyperbolic dynamics which are found in the theory of expansive geodesic flows in manifolds without conjugate points (see Theorem~\ref{expansivity} and compare~\cite{kn:Ruggiero97}). 
These results play a crucial role in the proof of the following result which is related to the thermodynamical formalism and ergodic optimization and gives a first  example of  possible applications of Theorem~\ref{quotient}.

\begin{thmy} \label{maximalentropy}
The geodesic flow of a  $C^{\infty}$ compact connected boundaryless surface without focal points and of genus greater than one has a unique (hence ergodic) invariant probability measure of maximal entropy.
\end{thmy}

Theorem \ref{maximalentropy} partially extends Knieper's result \cite{Kni:98} about the uniqueness of maximal entropy measures for  geodesic flows in rank one manifolds with nonpositive curvature. Though his result is valid in any dimension, we would like to point out that in the case of surfaces our setting is more general (there exist surfaces without focal points which admit some regions with positive curvature).

Our approach to show Theorem \ref{maximalentropy}  follows the classical thermodynamical formalism developed by Bowen \cite{Bow:74,Fra:77} for expansive topological dynamical systems in compact metric spaces. Since the geodesic flow in Theorem \ref{maximalentropy}  might not be expansive we combine Theorem \ref{quotient}  with some recent results by Buzzi \emph{et al.} \cite{BuzFisSamVas:12} which reduce the study of maximal entropy measures of a so-called extension of an expansive system to the study of maximal entropy measures of expansive ones (see Section~\ref{sec:entropystatement}). 

This natural, more topological, approach to study ergodic  invariant measures arising from topological dynamics seems promising regarding further generalizations under less (or no) restrictions on the curvature or Jacobi fields. We conjecture that Theorem \ref{quotient}  (and eventually Theorem \ref{maximalentropy}) extends to surfaces without conjugate points where Green bundles are continuous (see \cite{kn:Eberlein} for the definition), in particular to  so-called surfaces with bounded asymptote \cite{kn:Eschenburg}. Such surfaces might admit focal points and  in many respects be quite far from surfaces with nonpositive curvature. 

The main idea to prove Theorem \ref{quotient} is in many ways a very natural one: the geodesic flow of a compact surface without conjugate points and of higher genus is not expansive in general. It may have regions of non-expansivity, that in the case of a surface without focal points consist precisely of flat strips. So we define an equivalence relation in the
unit tangent bundle of the surface that identifies points in the same strip (see Section~\ref{stripproperties}, where further properties of such strips are studied).  This relation induces naturally a quotient flow which preserves the
classes and the time parametrization of the initial geodesic flow. The quotient space then carries a flow without ``non-expansive" orbits,
and the hard part of the proof consists in showing that the quotient space has a good topological structure. We show that the quotient space is
a topological 3-manifold, and hence carries a smooth manifold structure (see Section~\ref{quotientstructure}). Finally, we show that the quotient flow is expansive with respect to any metric defined on the quotient space (see Section~\ref{sec:dynquoflo}).

\section{Preliminaries}\label{sec:2}

\noindent\emph{Standing assumption.}
Throughout the paper $(M,g)$ will be a $C^{\infty}$ compact connected Riemannian manifold without boundary.
We shall always assume that $M$ has \emph{no conjugate points}, that is, that  at every point the exponential map is non-singular. In particular, we will study the particular subclass of manifolds \emph{without focal points}, that is, if $J(t)$ is a Jacobi field along a geodesic in $M$ with $J(0)=0$ then $\lVert J(t)\rVert$ is strictly increasing in $t$.
\medskip

Each vector $\theta\in TM$ determines a unique geodesic $\gamma_\theta(\cdot)$ such that $\gamma_\theta'(0)=\theta$. The geodesic flow $(\phi_t)_{t\in\bR}$ acts on $TM$ by $\phi_t(\theta)=\gamma_\theta'(t)$. We shall study its restriction to the unit tangent bundle $T_1M$, which is invariant. All geodesics will be parametrized by arc length.

We shall denote by $\tilde{M}$ the universal covering of $M$ and endow it with the pullback $\tilde g$ of the metric $g$ by the covering map $\pi\colon\tilde{M} \longrightarrow M$ which gives the Riemannian manifold $(\tilde{M}, \tilde{g})$.
For this manifold we  consider  also geodesics and the geodesic flow which acts on $T_1\tilde M$ and  we will also denote them by $\gamma_{\bar\theta}(\cdot)$ for given vector $\bar\theta\in T_1\tilde M$ and $(\phi_t)_{t\in\bR}$ (the domain of the flow $\phi_{t}$ is enough to specify the dynamical system under consideration),   respectively.
We  denote by $\bar{\pi} \colon T_{1}\tilde{M} \longrightarrow T_{1}M$ the natural projection.
The distance associated to the Riemannian metric $g$ will be denoted by $d_{g}$ and the one associated to $\tilde{g}$ by $d_{\tilde{g}}$. We will omit the metric and simply write $d$ if there is no danger of confusion. The Riemannian metric on $M$ lifts to the Sasaki metric on $TM$ which we denote by $d_S$. We shall use the same notation for the Sasaki distances in $T_{1}\tilde{M}$ and $\widetilde{T_{1}M}$. For any $\theta\in TM$ we will consider the orthogonal decomposition of $T_\theta TM$ into horizontal and vertical parts $T_\theta TM=H_\theta\oplus V_\theta$, similar for $\bar\theta\in T\tilde M$.

Two geodesics $\gamma_1$ and $\gamma_2$ in  $\tilde{M}$ are \emph{asymptotic} (as $t\to\infty$) if $d(\gamma_1(t),\gamma_2(t))$ is bounded as $t\to\infty$, that is, there exists $C>0$ such that $ d(\gamma_1(t),\gamma_2(t))\le C$ for all $t\ge0$, and \emph{bi-asymptotic} if $d(\gamma_1(t),\gamma_2(t))$ is bounded as $t\to\pm\infty$, that is, the previous inequality holds for all $t\in\bR$.

Given a metric space $(X,d)$, and two subsets $Z_1,Z_2 \subset X$, let us denote by $d_H(Z_1,Z_2)$ the Hausdorff distance between $Z_1$ and $Z_2$.

\subsection{Horospheres and invariant submanifolds}

A very special property of manifolds with no conjugate points is the
existence of the so-called Busemann functions:
given \( \bar\theta = (p,v) \in T_{1}\tilde{M} \), the \emph{forward} and \emph{backward Busemann functions} \( b_{\bar\theta}^\pm\colon \tilde{M} \longrightarrow \mathbb R \)
associated to  \( \bar\theta\) are defined by
\[
	b_{\bar\theta}^+(x) \eqdef \lim_{t\rightarrow +\infty}d(x,\gamma_{\bar\theta}(t)) - t \\
	\quad\text{ and }\quad
	b_{\bar\theta}^-(x) \eqdef \lim_{t\rightarrow +\infty}d(x,\gamma_{\bar\theta}(-t)) - t ,
\]
respectively.
The level sets of the Busemann functions are the \emph{horospheres}.
We define the (level $0$) \emph{positive} and \emph{negative horosphere} of $\bar\theta \in  T_{1}\tilde{M}$ by
\[
	H^+(\bar\theta)\eqdef (b_{\bar\theta}^+)^{-1}(0)
	\quad\text{ and }\quad
	H^-(\bar\theta)\eqdef (b_{\bar\theta}^-)^{-1}(0),
\]	
respectively.

Horospheres in $\tilde M$ lift naturally to horospheres in $T_1\tilde M$ as follows.
Consider the gradient vector fields $\nabla b_{\bar\theta}^\pm$ and define the \emph{positive} and \emph{negative horocycles} $\tilde{\mathscr F}^{\s/\u}(\bar\theta)$ in $T_1\tilde M$ through $\bar\theta$ to be the restriction of $\nabla b_{\bar\theta}^\pm$ to $H^\pm(\bar\theta)$
\[
	\tilde{\mathscr F}^\s(\bar\theta)
	\eqdef \big\{(q,-\nabla_{q}b_{\bar\theta}^+)\colon q \in H^+(\bar\theta)\big\}
\,\text{ and }\,	\tilde{\mathscr F}^\u(\bar\theta)
	\eqdef \big\{(q,-\nabla_{q}b_{\bar\theta}^-)\colon q \in H^-(\bar\theta)\big\},
\]
respectively.
Let us denote by \( \sigma^{\bar\theta}_{t} \colon \tilde{M} \longrightarrow \tilde{M} \) the integral flow of the vector field \( -\nabla b_{\bar\theta}^+ \).
The orbits of this flow are the {\em Busemann asymptotes} of $\gamma_{\bar\theta}$.

We will list some basic properties of horospheres that we shall need (see, for example,~\cite{kn:Pesin,kn:Eschenburg} for details).

\begin{lemma} \label{Busemann}
Let $(M,g)$ be a $C^{\infty}$ compact connected boundaryless Riemannian manifold without conjugate points. Then
\begin{enumerate}
\item [(i)]
\( b_{\bar\theta}^\pm\) are \( C^{1} \) functions for every $\bar\theta$.
\item [(ii)]
There exists a constant $K>0$ depending on the curvature of $M$ such that for given $\theta\in T_1\tilde M$, the gradient vector field $\nabla b_{\bar\theta}^\pm$ is $K$-Lipschitz and has unit length.  Each horosphere is an embedded submanifold of dimension $n-1$ tangent to a Lipschitz plane field.
\item [(iii)]
The orbits of \( \sigma^{\bar\theta}_{t}\) are geodesics which are everywhere perpendicular to the horospheres $H^+(\bar\theta)$.
In particular, the geodesic \( \gamma_{\bar\theta} \) is an orbit of this
flow and for every $t \in \mathbb R$ we have
\[
	\sigma^{\bar\theta}_{t}( H^+(\bar\theta))= H^+(\gamma_{\bar\theta}(t)) \,.
\]
\end{enumerate}
\end{lemma}

 Lemma~\ref{Busemann} item (iii) implies that the horospheres are equidistant, that is, given any point 
 $p \in H^+(\gamma_{\bar\theta}(t))$  for every $s$ we have $d(p, H^+(\gamma_{\bar\theta}(s)))=\vert t-s \vert$. Clearly \( H^+(\gamma_{\bar\theta}(t))\) varies continuously with $t \in \mathbb R$, however it is not known whether it varies continuously with \( \bar\theta\).

\subsection{Visibility, central manifolds, and topological dynamics}
The universal covering of a compact manifold without conjugate points belongs to a special class of manifolds without conjugate points satisfying the axiom of {visibility}. This geometric property was introduced by Eberlein and O'Neill~\cite{kn:EO} for manifolds without focal points (or even without conjugate points~\cite{kn:Eberlein}) as a criterion when a manifold behaves as if it would have negative curvature.

\begin{definition} 
A complete simply connected Riemannian manifold $(M,g)$ is a \emph{visibility manifold} if it has no conjugate points and if for every $\varepsilon >0$, $p \in M$ there exists $T=T(\varepsilon,p)>0$ such that for every two unit speed geodesic rays $\gamma_1$, $\gamma_2$ with $\gamma_1(0)=p= \gamma_2(0)$, if the distance from $p$ to every point of the geodesic joining $\gamma_1(t)$ to $\gamma_2(s)$, $0 \leq s \leq t$, is larger than $T$ then the angle formed by $\gamma_1'(0)$ and $\gamma_2'(0)$ is less than $\varepsilon$. When $T$ does not depend on $p$ we say that $(M,g)$ is a \emph{uniform visibility manifold}.
\end{definition}

If $(M,g)$ is compact and $(\tilde{M}, \tilde{g})$ is a visibility manifold, then $(\tilde M,\tilde{g})$  is a uniform visibility manifold. Moreover, if $(M,h)$ is another Riemannian structure without conjugate points and such that geodesic rays diverge in $\tilde M$, then $(\tilde{M}, \tilde{h})$ is also a visibility manifold (see \cite[Section 5]{kn:Eberlein}). Since Green~\cite{kn:Green} proved that in the case of a surface without conjugate points geodesic rays diverge, in particular, we obtain the following conclusion.

\begin{proposition}
	Let $(M,g)$ be a $C^\infty$ compact connected boundaryless Riemannian surface  without conjugate points and of genus greater than one. Then $(\tilde M,\tilde g)$ is a visibility manifold.
\end{proposition}

Given a point $\bar\theta \in T_{1}\tilde{M}$, its \emph{center stable set} and its \emph{center unstable set} are defined by
\[
	\tilde{\mathscr F}^{\cs}(\bar\theta)
	\eqdef
	\bigcup_{t \in {\mathbb R}} \tilde{\mathscr F}^\s\big(\phi_{t}(\bar\theta)\big)\\
	\quad\text{ and }\quad
	\tilde{\mathscr F}^{\cu}(\bar\theta)
	 \eqdef
	 \bigcup_{t \in {\mathbb R}} \tilde{\mathscr F}^\u\big(\phi_{t}(\bar\theta)\big)
	\,,
\]
respectively.
The images of $\tilde{\mathscr F}^{\s/\u}(\bar\theta)$  by the natural projection $\bar{\pi} \colon T_{1}\tilde{M} \longrightarrow T_{1}M$ are  the  \emph{stable} and the  \emph{unstable leaf} of $\theta=\bar{\pi}(\bar\theta)$ and denoted by $\mathscr F^{\s/\u}(\theta)$, respectively.
Likewise, the natural projections of $\tilde{\mathscr F}^{\cs/\cu}(\bar\theta)$ are the  \emph{center stable} and \emph{center unstable set} of $\theta$ and denoted by $\mathscr F^{\cs/\cu}(\theta)$, respectively.

We list now some important basic properties of the center stable and center unstable sets. For items (1)--(2) see~\cite[Section 6]{kn:Pesin},  the proof of item (4) is analogous to the proof of~\cite[Lemma 2.1]{BarRug:07}.
 
\begin{theorem} \label{central set}
	Let $(M,g)$ be a $C^{\infty}$ compact connected boundaryless Riemannian manifold without conjugate points such that $(\tilde{M}, \tilde{g})$ is a visibility manifold.
Then the following assertions hold:
\begin{enumerate}
\item [(i)] The family of sets
\[
	{\tilde {\mathscr  F}}^{\s}
	\eqdef \bigcup_{\bar\theta \in T_{1}\tilde {M}}{\tilde {\mathscr  F}}^{\s}(\bar\theta),\quad
	{\tilde {\mathscr  F}}^{\cs}
	\eqdef \bigcup_{\bar\theta \in T_{1}\tilde {M}}{\tilde {\mathscr  F}}^{\cs}(\bar\theta)
\]
are collections of $C^{0}$ submanifolds. In each union, these submanifolds  are either  disjoint
or coincide.

\item [(ii)] Given $p \in \tilde{M}$, there exists a homeomorphism
\[
	\Psi_{p}\colon \tilde{M} \times \tilde{V}_{p} \longrightarrow T_{1}\tilde{M}
\]	
such that
\[
	\Psi_{p}( \tilde{M}\times\{(p,v)\} ) = \tilde{\mathscr F}^{\cs}(p,v).
\]
In particular, ${\tilde {\mathscr  F}}^{\cs}$ and ${\mathscr  F}^{\cs}$ are continuous foliations and the space of leaves of ${\mathscr  F}^{\cs}$ is homeomorphic to the vertical
fiber $\tilde V_{p}$ for any $p \in \tilde{M}$.
\end{enumerate}
\end{theorem}

The foliations $\tilde{\mathscr F}^{\s/\cs}$ and  $\tilde{\mathscr F}^{\u/\cu}$ are called the \emph{stable}, \emph{central stable}, \emph{unstable}, and \emph{center unstable foliations}, respectively.
The image of such foliations by the covering map
$\bar{\pi} \colon T_{1}\tilde{M} \longrightarrow T_{1}M$ form foliations ${\mathscr  F}^{\s}$, ${\mathscr  F}^{\cs}$, ${\mathscr  F}^{\u}$, ${\mathscr  F}^{\cu}$ that we call by the same names.

The geodesic flow of a compact manifold whose universal covering is a visibility manifold shares many important dynamical properties with Anosov geodesic flows. The following results are proved by Eberlein \cite{kn:Eberlein} (see \cite{kn:Eberlein2} for visibility manifolds with nonpositive curvature),
and extend to visibility universal coverings of compact manifolds without conjugate points applying Gromov hyperbolic theory \cite{kn:Gromov}.
In particular, item (3) can be viewed as a sort of coarse local product structure (in general there is no global product structure for the geodesic flow of the universal covering of compact surfaces without conjugate points and of higher genus).

Recall that a continuous flow $\psi_{t}\colon X  \longrightarrow X$ is \emph{topologically mixing} if for any two open sets $U,V\subset X$ there exists $T>0$ such that for $\lvert t\rvert\ge T$, $\psi_t(U)\cap V\ne\emptyset$.
Recall that the flow $\psi_t$ is \emph{topologically transitive} if  for any open sets $U_1$ and $U_2$ there is $t>0$ such that $\psi_t(U_1)\cap U_2\ne\emptyset$ or, equivalently, if there exists a dense orbit.
Clearly every topologically mixing flow is topologically transitive.
Recall that a foliation is \emph{minimal} if every of its leaves is dense.

\begin{theorem} \label{transitivity}
	Let $(M,g)$ be a compact Riemannian manifold without conjugate points such that $(\tilde{M}, \tilde{g})$ is a visibility manifold.
Then:
\begin{enumerate}
\item [(i)] The foliations ${\mathscr  F}^{\s}$, ${\mathscr  F}^{\u}$ are minimal.\\[-0.4cm]
\item [(ii)] The geodesic flow of $(M,g)$ is topologically mixing.
\item [(iii)] Every two points $(p,v),(q,w) \in T_{1}\tilde{M}$ with $(q,w) \neq (p,-v)$ are \emph{heteroclinically related},
that is, we have
$$ \tilde{\mathscr  F}^{\cs}(p,v) \cap \tilde{\mathscr  F}^{\cu}(q,w) \neq \emptyset,\quad
 \tilde{\mathscr  F}^{\cs}(q,w) \cap \tilde{\mathscr  F}^{\cu}(p,v) \neq \emptyset .$$
\end{enumerate}
\end{theorem}

Finally recall that in the case that $M$ is a surface
 by a well-known result by Morse~\cite{kn:Morse} every minimizing geodesic in $(\tilde{M}, \tilde{g})$ is ``shadowed'' by a geodesic in the hyperbolic plane. Hence, if $(M,g)$ is a
$C^{\infty}$ compact  surface without focal points and of genus greater than one, then there exists $Q>0$ such that each minimizing geodesic in $(\tilde{M},\tilde{g})$ is contained in
the $Q$-tubular neighborhood of a certain hyperbolic geodesic. In particular, one can immediately conclude the following result.

\begin{proposition} \label{newMorseshadowing}
Let $(M,g)$ be compact Riemannian surface without conjugate points and of genus greater than one. Then there exists $Q=Q(M)>0$ such that the Hausdorff distance between any two bi-asymptotic geodesics in $\tilde M$ is bounded from above by $Q$.
\end{proposition}

\section{The geometry of strips}\label{stripproperties}

From now on we will always assume that $(M,g)$ is a $C^{\infty}$ compact connected Riemannian  surface without focal points and of genus greater than one.

\begin{definition}[Strip] Given $\bar\theta = (p,v) \in T_{1}\tilde{M}$, the \emph{strip} $F(\bar\theta) \subset \tilde{M}$
is the maximal subset of geodesics $\gamma_{\bar\eta} \subset \tilde{M}$ such that $\gamma_{\bar\eta}$ is an orbit of both the vector fields $-\nabla b_{\bar\theta}^+$ and $-\nabla b_{\bar\theta}^-$.
If the strip $F(\bar\theta)$ consists of a single geodesic, then the point $\bar\theta$ is called \emph{expansive} and the strip is called \emph{trivial}. Otherwise, we call the strip \emph{nontrivial}.
\end{definition}

In this section we survey geometric properties of strips. 
The starting point is the following Flat Strip Theorem (see 
\cite[Proposition 5.1]{kn:EO}, \cite[Theorem 7.3]{kn:Pesin}, note that it fails for manifolds without conjugate points \cite{Bur:92}).

\begin{theorem} \label{flatstrip}
	Any two bi-asymptotic geodesics in $(\tilde{M}, \tilde{g})$ are contained in a \emph{flat strip}, that is,  they are contained in an isometrically and totally geodesically embedded copy of an Euclidean strip $[a,b]\times\bR$ into the universal covering $\tilde M$.
\end{theorem}

\begin{definition}
	By Theorem \ref{flatstrip}, any strip $F(\bar\theta)$ is an isometric embedding $f\colon{\mathbb R}\times [a,b] \longrightarrow \tilde{M}$  into
the universal covering. If $[a,b]$ is chosen to be maximal with this property we call $b-a$ the \emph{width of the strip} $F(\bar\theta)$.
\end{definition}

An immediate consequence of Proposition \ref{newMorseshadowing} is the following.

\begin{lemma} \label{width}
	The width of any strip is bounded from above by $Q$, where $Q$ is the constant in Proposition \ref{newMorseshadowing}.
\end{lemma}

We gather in the next lemma the main geometric properties of nontrivial strips (see~\cite{kn:Pesin}).

\begin{lemma} \label{semicontinuity}
\begin{enumerate}
\item [(i)] The intersection of $F(\bar{\theta})$ with the horosphere $H^{+}(\bar{\theta})$ coincides with
$$
	I(\bar{\theta}) 
	:= H^{+}(\bar{\theta})\cap H^{-}(\bar{\theta}).
$$
\item [(ii)] For every $\bar{\theta} \in \tilde{M}$, the set $I(\bar{\theta})$ is a flat convex set, namely, if $p, q \in I(\bar{\theta})$
then the geodesic segment joining $p$ to $q$ also belongs to $I(\bar{\theta})$.
\end{enumerate}
\end{lemma}

The following lemma follows from~\cite[Lemma 2.1]{Bur:83}.

\begin{lemma} \label{lem:dense}
	If the orbit of a point in $T_{1}M$ by the geodesic flow is dense, then none of its lifts to $T_{1}\tilde{M}$ has a nontrivial strip.
\end{lemma}

In fact, we now have the following much stronger result (in the case of surfaces with non-positive curvature and genus greater than one due to Coudene and Schapira~\cite[Theorem 3.2]{CouSch:14} and in the general case to Schr\"oder~\cite{Sch:}, where the latter extends to an even much more general context).

\begin{lemma}
	Any nontrivial strip is periodic.
\end{lemma}

As a consequence, we immediately obtain the following.

\begin{corollary} \label{newexpgeneric}
	For any $\theta\in T_1M$ and any of its lifts $\bar\theta\in T_1\tilde M$, the set $\tilde{\mathscr F}^{\cs}(\bar\theta)$ contains at most one nontrivial strip. In particular, the set of nontrivial strips is countable and hence the set of expansive points is open and dense in $T_1\tilde M$.
\end{corollary}

\section{Quotient space  and the model flow}
\label{quotientstructure}

We now define as a key object in the present paper a relation in $T_{1}M$, and it is straightforward to check that it is indeed an equivalence relation.  It ``collapses'' each strip into a single curve. 

\begin{definition}\label{def:quotientflow}
Two points $\theta,\eta \in T_{1}M$ are related $\theta\sim\eta$ if, and only if,
\begin{itemize}
\item $\eta \in \mathscr F^\s(\theta)$\\[-0.3cm]
\item if $\bar\theta$ is any lift of $\theta$ and $\bar{\eta}$ a lift of $\eta$ satisfying $\bar{\eta} \in \tilde{\mathscr F}^\s(\bar{\theta})$,
then the geodesics $\gamma_{\bar{\theta}}$ and $\gamma_{\bar{\eta}}$ are bi-asymptotic.
\end{itemize}
Given $\theta\in T_1M$, we denote by $[\theta]$ the equivalence class which contains $\theta$. We denote by 
\[
	X\eqdef T_1M/_\sim
\]
 the set of all equivalence classes. Consider the quotient map
\[
	\chi \colon T_{1}M \longrightarrow X\colon \theta\mapsto [\theta].
\]		
\end{definition}

Recall that the \emph{quotient topology} associated to a quotient map $\chi$
 is the topology generated by the sets $U \subset X$ such that $\chi^{-1}(U)$ is an open set of $T_{1}M$.

We consider the flow $\psi_{t}\colon X \longrightarrow X$ defined by
\[
	\psi_{t}([\theta]) \eqdef [\phi_{t}(\theta)]\,.
\]
 As the geodesic flow preserves the foliation ${\mathscr  F}^\s$  and asymptoticity, this flow is indeed well defined. 
Moreover,  it is continuous in the quotient topology and  we have the following key fact which immediately follows from the very definition of the flows. 
  
\begin{proposition}\label{prop:semiconjugation}
	The geodesic flow $\phi_t\colon T_1M\to T_1M$ and the quotient flow $\psi_t\colon X\to X$ are time-preserving semi-conjugate by means of the quotient map  $\chi$.
\end{proposition}

The equivalence relation on $T_{1}M$ with quotient map $\chi$ induces naturally an equivalence relation in
$T_{1}\tilde{M}$ with quotient map $\bar{\chi} \colon T_{1}\tilde{M} \longrightarrow \bar{X}$.
Let us denote by $[\bar\theta]$ the corresponding equivalence class of $\bar\theta\in T_1\tilde M$. Let $\bar{\psi}_{t} \colon\bar{X} \longrightarrow \bar{X}$ be the corresponding quotient flow.%
\footnote{Notice that the absence of focal points was not needed in the above construction and that it extends to surfaces without conjugate points.}

 In the present section we will study the topological properties of the quotient space and this flow.
The main result of the section is the following.

\begin{theorem} \label{Hausdorff}
	Let $(M,g)$ be a compact surface without focal points. Then the quotient space $X$ is a compact topological 3-manifold. In particular, $X$ admits a smooth 3-dimensional structure where the quotient flow $\psi_t$ is continuous.
\end{theorem}

Theorem \ref{Hausdorff} is not at all obvious and requires a careful analysis of the quotient topology and its relationship with the dynamics of the geodesic flow.
The absence of focal points will be crucial in some subtle steps of the proof. The main idea of its proof is
to exhibit a special basis for the quotient topology, whose construction will be made in several steps.
The proof will be concluded at the end of this section.

\subsection{A family of cross sections and a basis for the quotient topology} \label{subs1}

As first step to obtain a basis for the quotient topology we construct a special family of cross sections
for the quotient flow from which we shall obtain a basis by shifting them by the geodesic flow.

Given $\bar{\theta} \in T_{1}\tilde{M}$, let
$$
	{\mathcal I}(\bar{\theta})
	\eqdef \tilde{\mathscr F^\s}(\bar{\theta}) \cap \tilde{\mathscr F^\u}(\bar{\theta}).
$$
Note that the set $\mathcal I(\bar\theta)$ is a lift of $I(\bar{\theta})$ (defined in Lemma~\ref{semicontinuity}) to $T_{1}\tilde{M}$. Moreover, the set  $\tilde{\mathscr F}^{\cs}(\bar{\theta}) \cap \tilde{\mathscr F}^{\cu}(\bar{\theta})$ contains an isometric copy of the  strip $F(\bar\theta)$ (which can be trivial or nontrivial).

\begin{figure}
\begin{overpic}[scale=.48 
  ]{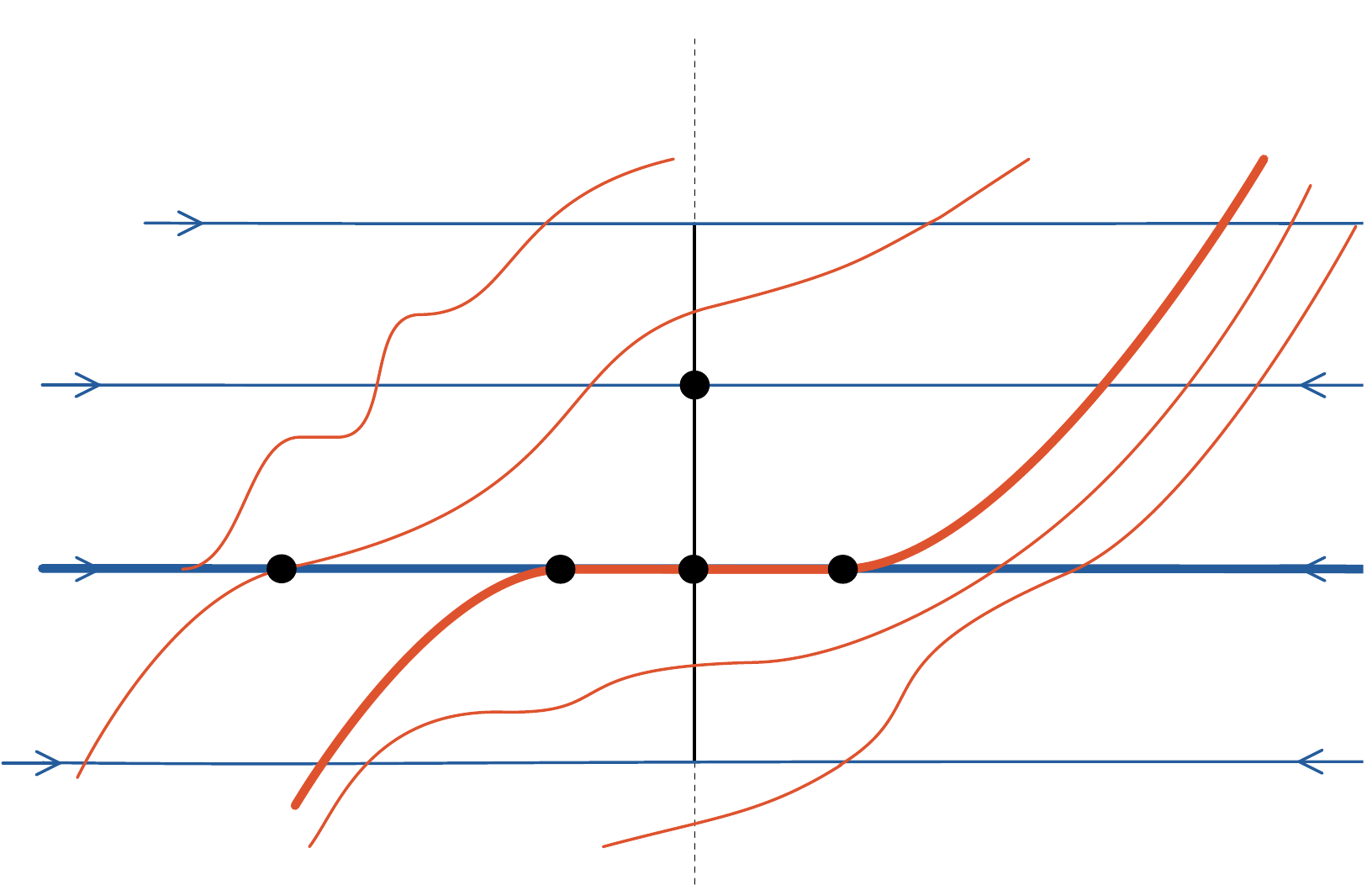}
      \put(35,26){\tiny$R(r_0^-,0)$}
      \put(57,26){\tiny$R(r_0^+,0)$}
      \put(52,39){\tiny$R(0,s)$}
      \put(17,26){\tiny$R(r,0)$}
      \put(-9,22){\tiny$\tilde{\mathscr F^\s}(\bar{\theta})$}
      \put(-18,36){\tiny$\tilde{\mathscr F^\s}(R(0,s))$}
      \put(86,55){\tiny$\tilde{\mathscr F^{\cu}}(\bar{\theta})\cap \Sigma$}
      \put(48,19){\tiny$\bar{\theta}=R(0,0)$}
      \put(51,50){\tiny$V_\delta(\bar\theta)$}
\end{overpic}
\caption{Parametrization of the disk $\Sigma=\Sigma_{\bar\theta}(\varepsilon,\delta)$}
\label{fig:1}
\end{figure}

Let us now choose the local cross section. Given a point $\bar{\theta}$, let us construct a smoothly embedded closed two-dimensional disc $\Sigma=\Sigma_{\bar\theta}(\varepsilon,\delta)\subset T_{1}\tilde{M}$ which is transverse to the geodesic flow and which contains $\mathcal I(\bar\theta)$. This disk will be foliated by the leaves of $\tilde{\mathscr F}^\s$.
To begin the construction, let $\varepsilon >0$, $\delta >0$ be sufficiently small, let $V_{\delta}(\bar{\theta})$ be the $\delta$-tubular neighborhood of
$\bar{\theta}$ in its vertical fiber and let
\[
	R\colon
	(r_0^--\varepsilon, r_0^++\varepsilon)\times (-\delta, \delta) \longrightarrow 
	T_1\tilde M
\]
be the homeomorphism with the following properties:
\begin{itemize}
\item $R(0,0) = \bar{\theta}$,
\item $R(0,s)$ with $s \in (-\delta, \delta)$ is the arc length parametrization of $V_{\delta}(\bar{\theta})$ in the Sasaki metric.
\item $R(r,0)$ is the arc length parametrization of the $\varepsilon$-tubular neighborhood of ${\mathcal I}(R(0,0)) ={\mathcal I}(\bar{\theta})$ in  $\tilde{\mathscr F}^\s(\bar{\theta})$,
$R(r_0^-,0)$ and $R(r_0^+,0)$ being the endpoints of ${\mathcal I}(\bar\theta)$;
\item For each $s \in (-\delta, \delta)$, $r\mapsto R(r,s)$ is the arc length parametrization of the continuous curve $R(\cdot,s)$ in $\tilde{\mathscr F}^\s(R(0,s))$.
\end{itemize}
Since the foliation $\tilde{\mathscr F}^\s$ is a continuous foliation by Lipschitz curves,
by Brower's Open Mapping Theorem the image of $R$ is a two-dimensional section that we will denote by $\Sigma_{\bar{\theta}}(\varepsilon,\delta)$,
\[
	\Sigma_{\bar\theta}(\varepsilon,\delta)
	:= R\big((r_0^--\varepsilon, r_0^++\varepsilon)\times (-\delta, \delta)\big)
\]
(compare Figure~\ref{fig:1}).
Given $\tau>0$, let us consider the open neighborhood of $\bar{\theta}$ in $T_{1}\tilde{M}$ of the form
\begin{equation}\label{def:Bball}
	B_{\bar{\theta}}(\varepsilon, \delta, \tau)
	\eqdef \big\{ \phi_{t}\big(\Sigma_{\bar{\theta}}(\varepsilon,\delta)\big)
				\colon \lvert t \rvert <\tau \big\}
\end{equation}
where $\phi_{t}$ is the geodesic flow of $T_{1}\tilde{M}$.
We shall also omit the point $\bar{\theta}$ in the index unless we change it.
Clearly, the section depends on the point $\bar\theta$ and on the parameters $\varepsilon$ and $\delta$,
we shall omit this dependence in the notation but keep it in mind.
Let
\[
	\Pi_{\Sigma} \colon
	B(\varepsilon, \delta, \tau)	\to \Sigma
\]	
be the projection from $B(\varepsilon,\delta,\tau):=B_{\bar\theta}(\varepsilon,\delta,\tau)$ onto $\Sigma=\Sigma_{\bar\theta}(\varepsilon,\delta)$ by $\phi_t$.

We have the following result (compare Figure~\ref{fig:2}) essentially saying that $\chi(\Sigma)$ is ``almost" a cross section for the quotient flow.

\begin{lemma}
Given a strip $ F$ which intersects $B(\varepsilon, \delta, \tau)$,
there exists $\bar\eta \in \Sigma$ such that $\Pi_{\Sigma}( F\cap B(\varepsilon, \delta, \tau))$ is a connected component of  ${\mathcal I}(\bar\eta)$. In particular, if ${\mathcal I}(\bar\eta) \subset  B(\varepsilon, \delta, \tau)$ then
$\Pi_{\Sigma}( F) = {\mathcal I}(\bar\eta)$ and
$$
	 F \cap B(\varepsilon, \delta, \tau)
	= \big\{\phi_{t}({\mathcal I}(\bar\eta))\colon\lvert t\rvert<\tau \big\}
	\,.
$$
\end{lemma}

\begin{proof}
Observe that, by the construction of $\Sigma$, each strip $F$ intersects $\Sigma$ 
 in a connected component of some ${\mathcal I}(\bar\eta)$, $\bar\eta \in \Sigma$.
 Moreover, the geodesic flow preserves the sets ${\mathcal I}$, that is, for every $t \in {\mathbb R}$ and  for every $\bar\eta \in T_{1}\tilde{M}$ we have $\phi_{t}({\mathcal I}(\bar\eta))= {\mathcal I}(\phi_{t}(\bar\eta))$.
This gives the claim.
\end{proof}

\begin{figure}
\begin{overpic}[scale=.35,
  ]{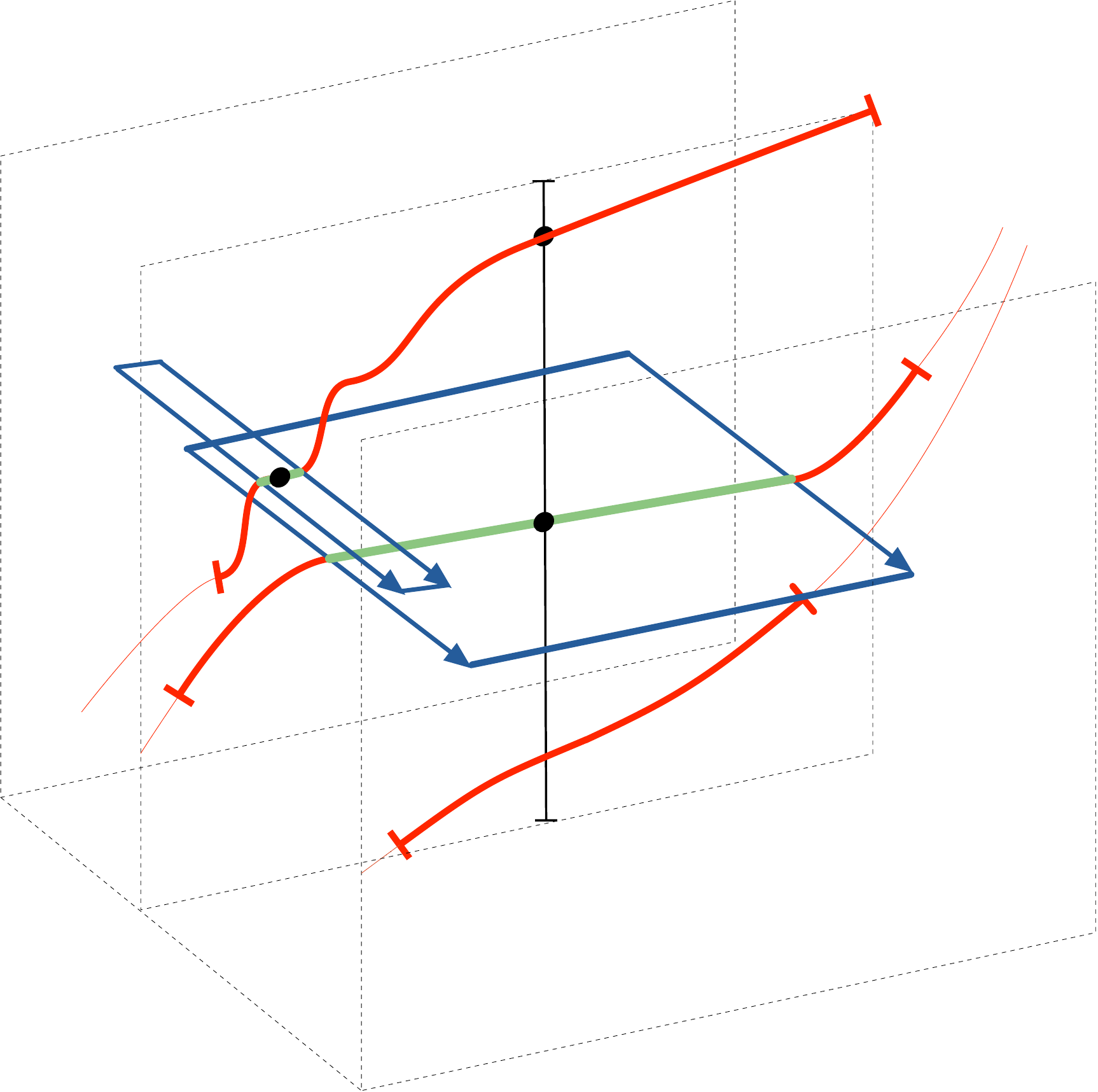}
       \put(6,64){\tiny$ F$}
       \put(28,55){\tiny$\bar\eta$}
       \put(51,50){\tiny$\bar\theta$}
\end{overpic}
\caption{Flat strips (blue) in $B_{\bar{\theta}}(\varepsilon,\delta,\tau)$ and their projections (green) to $\Sigma_{\bar{\theta}}(\varepsilon,\delta)$}
\label{fig:2}
\end{figure}

Let us denote
\[ \begin{split}
	\Sigma^{+}&\eqdef \{R(r,s)\colon r\in(r_0^--\varepsilon,r_0^++\varepsilon),
		s\in(0,\delta)\}\subset\Sigma,\\
	\Sigma^{-}&\eqdef \{R(r,s)\colon  r\in(r_0^--\varepsilon,r_0^++\varepsilon),
		s\in(-\delta,0)\}\subset\Sigma
	\,.
\end{split}\]
For each point $\bar{\eta} \in B(\varepsilon, \delta, \tau)$, denote by
\[
	B^{\cs}_{\bar\eta}(\varepsilon,\delta,\tau)\subset
	\tilde{\mathscr F}^{\cs}(\bar{\eta}) \cap B(\varepsilon, \delta, \tau),\quad
	B^{\cu}_{\bar\eta}(\varepsilon,\delta,\tau)\subset
	\tilde{\mathscr F}^{\cu}(\bar{\eta}) \cap B(\varepsilon, \delta, \tau)
\]	
the connected components of the intersections of the central stable and unstable sets of $\bar{\eta}$ with $B(\varepsilon, \delta, \tau)$
which contain $\bar{\eta}$, respectively. Given $\bar{\eta} \in \Sigma$ 
let
\begin{equation}\label{eq:defWSU}\begin{split}
 W^\s_\Sigma(\bar{\eta})
	&\eqdef
	 \Pi_{\Sigma}\big(B^{\cs}_{\bar\eta}(\varepsilon,\delta,\tau)\big)
	= B^{\cs}_{\bar\eta}(\varepsilon,\delta,\tau)) \cap \Sigma_{\bar{\theta}}(\varepsilon,\delta)
	,\\
	 W^\u_\Sigma(\bar\eta)
	&\eqdef
	 \Pi_{\Sigma}\big(B^{\cu}_{\bar\eta}(\varepsilon,\delta,\tau)\big)\,.
\end{split}\end{equation}
Note that, in fact, by the definition of the map $R$, for every $\bar{\eta} \in \Sigma$  there exist parameters $r,s$ such that $W^{\s}_\Sigma(\bar{\eta}) = R(r,s)$.
However, in general $W^{\u}_\Sigma(\bar{\eta})$ may not satisfy such a property.

If the section $\Sigma$ is sufficiently narrow in the vertical direction being close enough to ${\mathcal I}(\bar{\theta})$ then every two different points in $\Sigma$ are heteroclinically related.
Given $\bar\theta$, choose $\delta_0=\delta_0(\bar\theta)>0$ small so that $\tilde{\mathscr F}^{\cu}(\bar\theta)$ intersects $\tilde{\mathscr F}^\s(R(0,\pm\delta_0))$.
Moreover, for $\delta\in(0,\delta_0)$ there exist $\varepsilon_0=\varepsilon_0(\bar\theta,\delta)>0$ such that for every $\varepsilon\in(0,\varepsilon_0)$ for every $r \in (r_0^- -\varepsilon,r_0^-)\cup (r_0^+, r_0^++\varepsilon)$ and every $s$ with $\lvert s\rvert <\delta$ we have
$$
	 W^\u_\Sigma(R(r,0)) \cap  W^\s_\Sigma(R(0,s) )
	\subset  \Sigma_{\bar{\theta}}(\varepsilon,\delta)
	\,.
$$

The basis we will construct is in many respects a ``blow up'' of the classical
local product neighborhood in hyperbolic dynamics (see for example~\cite[Chapter 6.4]{KatHas:95}).

\begin{definition}
Given $\bar\eta,\bar\xi\in\Sigma_{\bar\theta}(\varepsilon,\delta)$,  let
\[
	[\bar{\eta}, \bar{\xi}\,]\eqdef
	 W^\s_\Sigma(\bar{\eta})\cap  W^\u_\Sigma(\bar{\xi})\,.
\]
\end{definition}

As a consequence of Corollary \ref{newexpgeneric}, given $\varepsilon\in(0,\varepsilon_0)$ there exist numbers $\rho^-,\rho^+\in(0,\varepsilon)$ such that 
\[
	\bar\theta^{-}= R(r_0^- -\rho^-,0)
	\quad\text{ and }\quad
	\bar\theta^{+} = R(r_0^+ + \rho^+,0)
\]	
are expansive points and hence the sets $W^\u_\Sigma(\bar\theta^\pm)$ are curves which are disjoint from $W^\u_\Sigma(\bar{\theta})$. Moreover,
all such curves bound a region 
in $\Sigma_{\bar{\theta}}(\varepsilon,\delta)$ which is homeomorphic to a rectangle whose relative interior contains ${\mathcal I}(\bar{\theta})$ (compare Figure~\ref{fig:4}).
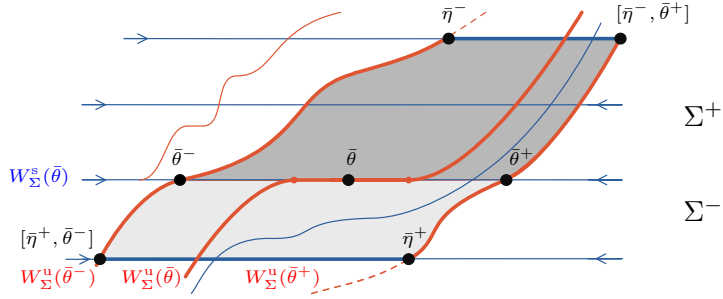
\begin{figure}
\begin{overpic}[scale=.45,%
  ]{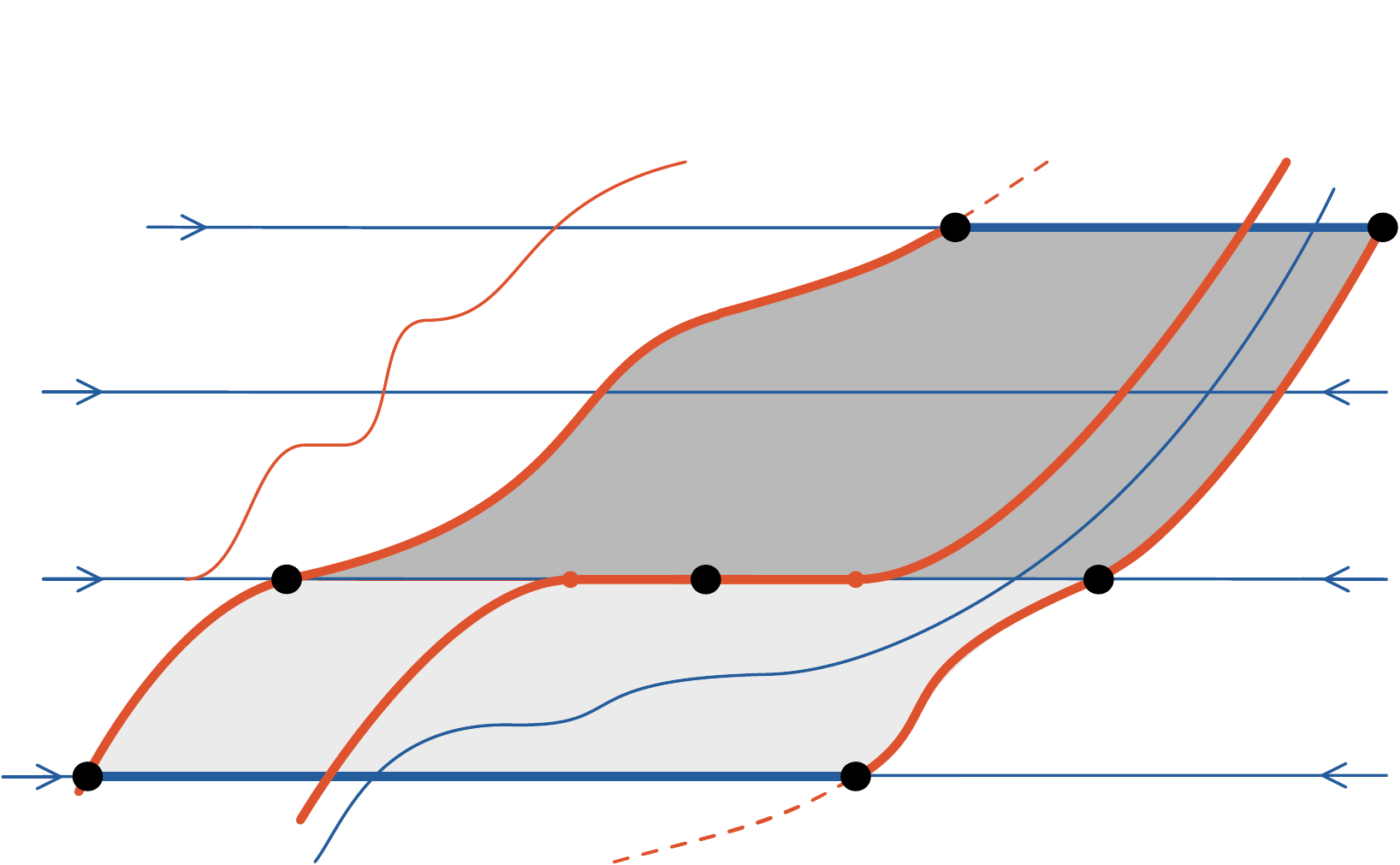}
      \put(-8,9){\tiny$[\bar\eta^{+},\bar\theta^{-}]$}
      \put(98,49){\tiny$[\bar\eta^{-},\bar\theta^{+}]$}
      \put(60,9){\tiny$\bar\eta^{+}$}
      \put(19,23){\tiny$\bar\theta^{-}$}
      \put(79,23){\tiny$\bar\theta^{+}$}
      \put(-8,2){\tiny\textcolor{red}{$W^\u_\Sigma(\bar\theta^{-})$}}
      \put(32,2){\tiny\textcolor{red}{$W^\u_\Sigma(\bar\theta^{+})$}}
      \put(10,2){\tiny\textcolor{red}{$W^\u_\Sigma(\bar{\theta})$}}
      \put(-10,20){\tiny\textcolor{blue}{$W^\s_\Sigma(\bar{\theta})$}}
      \put(50,23){\tiny$\bar{\theta}$}
      \put(67,49){\tiny$\bar{\eta}^-$}
      \put(110,30){$\Sigma^+$}
      \put(110,13){$\Sigma^-$}
\end{overpic}
\caption{Region defined by expansive points $\bar\theta^{\pm}$ and $\bar\eta^\pm$, contained in the region $\Sigma_{\bar{\theta}}(\varepsilon,\delta)$ splits into $\Sigma^+$ and $\Sigma^-$, and containing the open set $U_{\bar\theta}(\varepsilon,\delta,\bar\theta^-,\bar\theta^+,\bar\eta^-,\bar\eta^+)$ (shaded region)}
\label{fig:4}
\end{figure}
The set $W^\s_\Sigma(\bar{\theta})$ divides $W^\u_\Sigma(\bar\theta^{-})$ into two parts: one in $\Sigma^{+}$ and one in $\Sigma^{-}$.
By Theorem~\ref{transitivity} item (iii) we have
\[
	\bar\eta^{-} \in W^\u_\Sigma(\bar\theta^{-})\cap \Sigma^{+}
	\quad\text{ and }\quad
	\bar\eta^{+} \in  W^\u_\Sigma(\bar\theta^{+})\cap \Sigma^{-}
\]	
and the intersections $[\bar\eta^{-},\bar\theta^{+}]$ and  $[\bar\eta^{+},\bar\theta^{-}]$ are nonempty (though, may be contained in a nontrivial strip), and by the previous remarks we can assume that such intersections are in $\Sigma_{\bar{\theta}}(\varepsilon,\delta)$.
Fixing any such points $\bar\eta^-$ and $\bar\eta^+$, let us denote by $U_{\bar\theta}(\varepsilon,\delta, \bar\theta^-,\bar\theta^+,\bar\eta^-,\bar\eta^+)$ the open two-dimensional region
in $\Sigma_{\bar{\theta}}(\varepsilon,\delta)$ whose boundary is formed by the described pieces of stable and unstable arcs. 
The region $U_{\bar\theta}(\varepsilon,\delta , \bar\theta^-,\bar\theta^+,\bar\eta^-,\bar\eta^+)$ clearly contains $ {\mathcal I}(\bar\theta)$.

\begin{lemma} \label{boundarycurve}
	Given $\bar{\theta} \in T_{1}\tilde{M}$, let $\delta_0=\delta_0(\bar\theta)$, $\delta\in(0,\delta_0)$, $\varepsilon_0=\varepsilon_0(\bar\theta,\delta)$, $\varepsilon\in(0,\varepsilon_0)$,  expansive points $\bar{\theta}^{-}=R(r_0^--\rho^-,0)$, $\bar{\theta}^{+}=R(r_0^++\rho^+,0)$ with $\rho^\pm\in(0,\varepsilon)$, and points $\bar\eta^\pm\in W^\u_\Sigma(\bar\theta^\pm)\cap\Sigma^\mp$, the above constructed region $U=U_{\bar\theta}(\varepsilon,\delta,\bar{\theta}^{-},\bar{\theta}^{+},\bar\eta^-,\bar\eta^+)$ in the section $\Sigma_{\bar\theta}(\varepsilon,\delta )$ has the following properties:
\begin{enumerate}
\item [(i)] We have $\bar{\chi}^{-1}(\bar{\chi}(U))=U$
	and hence the set $\bar{\chi}(U)$ 
	is an open neighborhood of $\bar{\chi}(\bar{\theta})$ in the quotient topology restricted to $\bar{\chi}(\Sigma_{\bar\theta}(\varepsilon,\delta ))$.\\[-0.3cm]
\item [(ii)]
For every positive numbers $t',\delta',\varepsilon'$ we can choose $\delta,\varepsilon$ such that the above considered region $U=U_{\bar\theta}(\varepsilon,\delta,\bar{\theta}^{-},\bar{\theta}^{+},\bar\eta^-,\bar\eta^+)$  for every $t$ with $\lvert t\rvert\le t'$ satisfies
\[
	\phi_t(U)=
	U_{\phi_t(\bar\theta)}\big(\varepsilon, \delta,
		\phi_t(\bar\theta^{-}),\phi_t(\bar\theta^{+}),
		\phi_t(\bar\eta^-),\phi_t(\bar\eta^+)\big)
	\subset
	\Sigma_{\phi_{t}(\bar{\theta})}(\varepsilon',\delta'),
\]	
and moreover,
\[
	 \phi_t(U)
	 \subset
	 U_{\phi_t(\bar\theta)}\big(\varepsilon, \delta',
	 \phi_t(\bar\theta^-), \phi_t(\bar\theta^+),\phi_t(\bar\eta^-),\phi_t(\bar\eta^+)\big).
\]	
\item [(iii)] The set $\bar{\chi}(U)$ is a (topological) local cross section of the quotient flow in $\bar{X}$, that is, there exists an open set containing $\bar\chi(\bar\theta)$ such that every orbit of the quotient flow restricted to this set intersects $\bar\chi(U)$ in just one point.
\end{enumerate}
\end{lemma}

\begin{proof}
Given $\bar{\xi} \in U$, we have $\bar{\chi}^{-1}(\bar{\chi}(\bar{\xi}))={\mathcal I}(\bar{\xi}) \subset \Sigma_{\bar\theta}(\varepsilon,\delta)$. By the construction of $ \Sigma_{\bar\theta}(\varepsilon,\delta)$, the curve ${\mathcal I}(\bar{\xi})$ meets the boundary $C$ of $U$ if, and only if, $\bar{\xi}$ is in $C$ already. Because this boundary  is made of pieces of center stable and center unstable leaves, if a strip through a point $\bar{\xi} \in  \Sigma_{\bar\theta}(\varepsilon,\delta)$ meets $C$ then the whole set ${\mathcal I}(\bar{\xi})$ must be included in one of these pieces of center stable and center unstable leaves.
So we conclude $\bar{\chi}^{-1}(\bar{\chi}(\bar{\xi}))\in U$ for every
$\bar{\xi} \in U$. Since clearly $\bar{\chi}^{-1}(\bar{\chi}(U))\supset U$ we have shown $\bar{\chi}^{-1}(\bar{\chi}(U))=U$.

By the definition of the quotient topology restricted to $ \Sigma_{\bar\theta}(\varepsilon,\delta)$, we have that $\bar{\chi}(U)$ is a relative open neighborhood of $\bar{\chi}(\bar{\xi})$ in $\bar{\chi}( \Sigma_{\bar\theta}(\varepsilon,\delta))$ thus proving item (1) in the lemma.

The proof of item (2) follows from the construction of $U$. Indeed, all the dynamical objects involved in its construction, that is, stable leaves and heteroclinic intersections, are invariant by the geodesic flow. The constants $\varepsilon$ and $\delta $
may vary a little since they are geometric quantifiers of compact pieces of stable leaves which contain strips. The size of a strip does not
change under the action of the geodesic flow but the size of a neighborhood of it changes continuously. From the above statements is
straightforward to conclude item (2).

Item (3) follows from the construction by the definition of equivalence relation within any strip.
\end{proof}

\begin{lemma} \label{basis}
	Given $\bar{\theta} \in T_{1}\tilde{M}$,  $\delta_0=\delta_0(\bar\theta)$, $\delta\in(0,\delta_0)$, $\varepsilon_0=\varepsilon_0(\bar\theta,\delta)$, $\varepsilon\in(0,\varepsilon_0)$, expansive points $\bar{\theta}^{-}=R(r_0^--\rho^-,0)$, $\bar{\theta}^{+}=R(r_0^++\rho^+,0)$ with $\rho^\pm\in(0,\varepsilon)$, and points $\bar\eta^\pm\in W^\u_\Sigma(\bar\theta^\pm)\cap\Sigma^\mp$%
, consider the above constructed region $U=U_{\bar\theta}(\varepsilon,\delta,\bar{\theta}^{-},\bar{\theta}^{+},\bar\eta^-,\bar\eta^+)$ in the section $\Sigma_{\bar\theta}(\varepsilon,\delta )$. Then for $\tau>0$ small enough  the set
\begin{equation}\label{def:setA}
	A=A_{\bar{\theta}}(\tau,\varepsilon,\delta,
			\bar\theta^-,\bar\theta^+,\bar\eta^-,\bar\eta^+)
	\eqdef
	\bigcup_{\lvert t \rvert <\tau} \phi_{t}(U)
\end{equation}
satisfies $\bar{\chi}^{-1}(\bar{\chi}(A))= A$.

Hence, the collection of such sets
\[
 	\Big\{\bar{\chi}\big(A_{\bar{\theta}}(\tau,\varepsilon,\delta,
			\bar{\theta}^{-},\bar{\theta}^{+},\bar\eta^-,\bar\eta^+)\big)\Big\}
\]
forms a basis for the quotient topology of $\bar{X}$.
\end{lemma}

\begin{proof}
First observe that $A$ by definition is homeomorphic to $(-\tau, \tau) \times U$. Therefore, Brower's Open Mapping Theorem implies that each such set is open in $T_{1}\tilde{M}$.

Moreover, by Lemma \ref{boundarycurve} item (ii) we have
\[
	 A = \bigcup_{\lvert t \rvert <\tau}
	 \phi_t(U)=\bigcup_{\lvert t\rvert<\tau}
	 U_{\phi_{t}(\bar{\theta})}\big(\varepsilon,\delta,
	 		\phi_t(\bar\theta^-), \phi_t(\bar\theta^+),
			\phi_t(\bar\eta^-),\phi_t(\bar\eta^+))\big)
\]	
for $\tau$ small enough. Applying Lemma \ref{boundarycurve} item (i) to the above union of sets we deduce that $\bar{\chi}^{-1}(\bar{\chi}(A)) = A$ as claimed.

This yields that the family of all such sets $A_{\bar{\theta}}(\tau,\varepsilon, \delta ,\bar{\theta}^{-},\bar{\theta}^{+},\bar\eta^-,\bar\eta^+)$ specified as in the statement are open sets in the quotient $\bar{X}$ since, by definition, an open set for the quotient topology is any set whose preimage by the quotient map $\bar{\chi}$ is an open set in $T_{1}\tilde{M}$.
To see that they provide a basis for the quotient topology observe that the family of open sets $B_{\bar{\theta}}(\varepsilon, \delta, \tau')$ defined in~\eqref{def:Bball} is a basis
for the family of open neighborhoods of the sets ${\mathcal I}(\bar{\theta})$, and that $A_{\bar{\theta}}(\tau,\varepsilon,\delta,\bar{\theta}^{-},\bar{\theta}^{+},\bar\eta^-,\bar\eta^+)$ is contained in $B_{\bar{\theta}}(\varepsilon, \delta, \tau')$. So let $\bar{\chi}(\bar{\theta})$ be the equivalence class of $\bar{\theta}$ and
$V$ an open neighborhood of $\bar{\chi}(\bar{\theta})$ in the quotient $\bar{X}$. The set $\bar{\chi}^{-1}(V)$ is an open neighborhood of
${\mathcal I}(\bar{\theta})$ which contains some $B_{\bar{\theta}}(\varepsilon, \delta, \alpha)$ for some parameters $\varepsilon, \delta, \tau'$.
Hence the family of sets $A_{\bar{\theta}}(\tau, \rho, \delta, \bar{\theta}^{-},\bar{\theta}^{+},\bar\eta^-,\bar\eta^+)$ provides a basis, too.
\end{proof}

\subsection{Smooth manifold structure}

\begin{proposition} \label{productstructure}
Given $\bar{\theta} \in T_{1}\tilde{M}$, let $U=U_{\bar{\theta}}(\varepsilon, \delta,\bar{\theta}^{-},\bar{\theta}^{+},\bar\eta^-,\bar\eta^+)$  be the local cross section of the geodesic flow as defined  in the previous section and consider the corresponding set $A=A_{\bar{\theta}}(\tau,\varepsilon,\delta,
			\bar\theta^-,\bar\theta^+,\bar\eta^-,\bar\eta^+)$ as defined in~\eqref{def:setA} for some $\tau>0$. There exist numbers $a<a'$, $ b< b'$  depending on $\bar\theta,\delta,\varepsilon,\bar{\theta}^{-},\bar{\theta}^{+},\bar\eta^-,\bar\eta^+$ and  a homeomorphism
$$
	f\colon
	(a, a') \times (b,b') \times (-\tau,\tau)
	\longrightarrow
	\bigcup_{\lvert t \rvert <\tau}
	\bar{\psi}_{t}\big(\bar\chi( U)\big)
	= \bar\chi( A)
$$
for every $\tau>0$.

In particular, the quotient spaces $\bar{X}$ and $X$ are topological 3-manifolds.
\end{proposition}

\begin{proof}
The proof relies essentially on the transitivity of the geodesic flow and the minimality of central foliations. By Theorem \ref{transitivity} item (i), each stable (unstable) leaf is dense in $U$. In particular, this holds for the central stable (central unstable) leaf of a dense orbit which by Lemma~\ref{lem:dense} has no nontrivial strips. 

Let $\theta^\ast\in T_{1}M$ be a point whose orbit is dense in $T_1M$, and let $\bar\theta^\ast$ one of its lifts in $T_{1}\tilde{M}$. Suppose that $\bar\theta^\ast \in U$ and that $d_{S}(\bar\theta^\ast, \bar{\theta}) <\delta$. The lifts of the stable set and the unstable set of the orbit of $\bar\theta^\ast$ in $T_{1}\tilde{M}$ are both dense in $\Sigma=\Sigma_{\bar{\theta}}(\varepsilon,\delta)$.

Consider the arc length parameterizations
\[\begin{split}
	&R^{\s}\colon (a,a') \longrightarrow
		U\quad\text{ with }\quad
	R^\s(0)=\bar\theta^\ast\\
	&R^{\u}\colon (b,b') \longrightarrow
		U\quad\text{ with }\quad
	R^\u(0)=\bar\theta^\ast
\end{split}\]	
of the arc of the intersection of the stable leaf $\tilde{\mathscr F}^\s(\bar\theta^\ast)$ with $U$ and  of the  arc of  intersection of the center unstable leaf
 $\tilde{\mathscr F}^{\cu}(\bar\theta^\ast)$  with $U$, respectively. Since we can choose $\bar\theta^\ast$ as closed to $\bar{\theta}$ as we wish, we can choose $a'-a$  very close to the length of the connected component  of the stable set of $\bar{\theta}$ in $U$ containing $\bar{\theta}$.
Analogously, $b' - b$ can be chosen very close to the length of the connected component of the unstable set of $\bar{\theta}$ in $U$ containing $\bar{\theta}$.

\begin{claim} 
	For every $r \in (a,a')$, $s \in (b,b')$ the point $[R^\u(r), R^\s(s)]$ is contained in $U$.
	 Moreover, the map
\[
h\colon (a,a') \times (b,b') \longrightarrow
\bar{\chi}(U)\colon
(r,s) \mapsto \bar{\chi}\big([R^\u(r), R^\s(s)]\big)
\]	
is a homeomorphism. In particular, $\bar{\chi}(U)$ is a topological 2-manifold.
\end{claim}

\begin{proof}
Recall the definition of $W^{\s/\u}_\Sigma(\cdot)$ in~\eqref{eq:defWSU}.
The set $U$ is foliated by
\[
	\{ W^\s_\Sigma(\bar{\eta})\cap U\colon\bar\eta\in U\}
	\quad\text{ and  }\quad
	\{ W^\u_\Sigma(\bar{\eta})\cap U\colon\bar\eta\in U\}
\]	
in a trivial way.
Both families of sets induce a product foliation of $U$ by embeddings of open segments.  Therefore, each of these sets splits its complement in the closure of $U$ into
two disjoint open regions (see Figure~\ref{fig:5}).
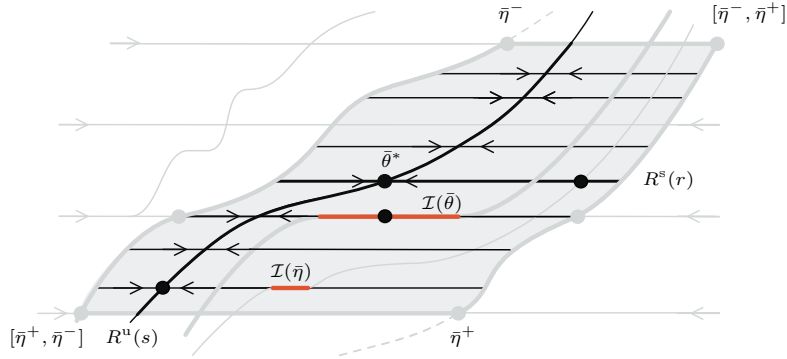
\begin{figure}
\begin{overpic}[scale=.55,
  ]{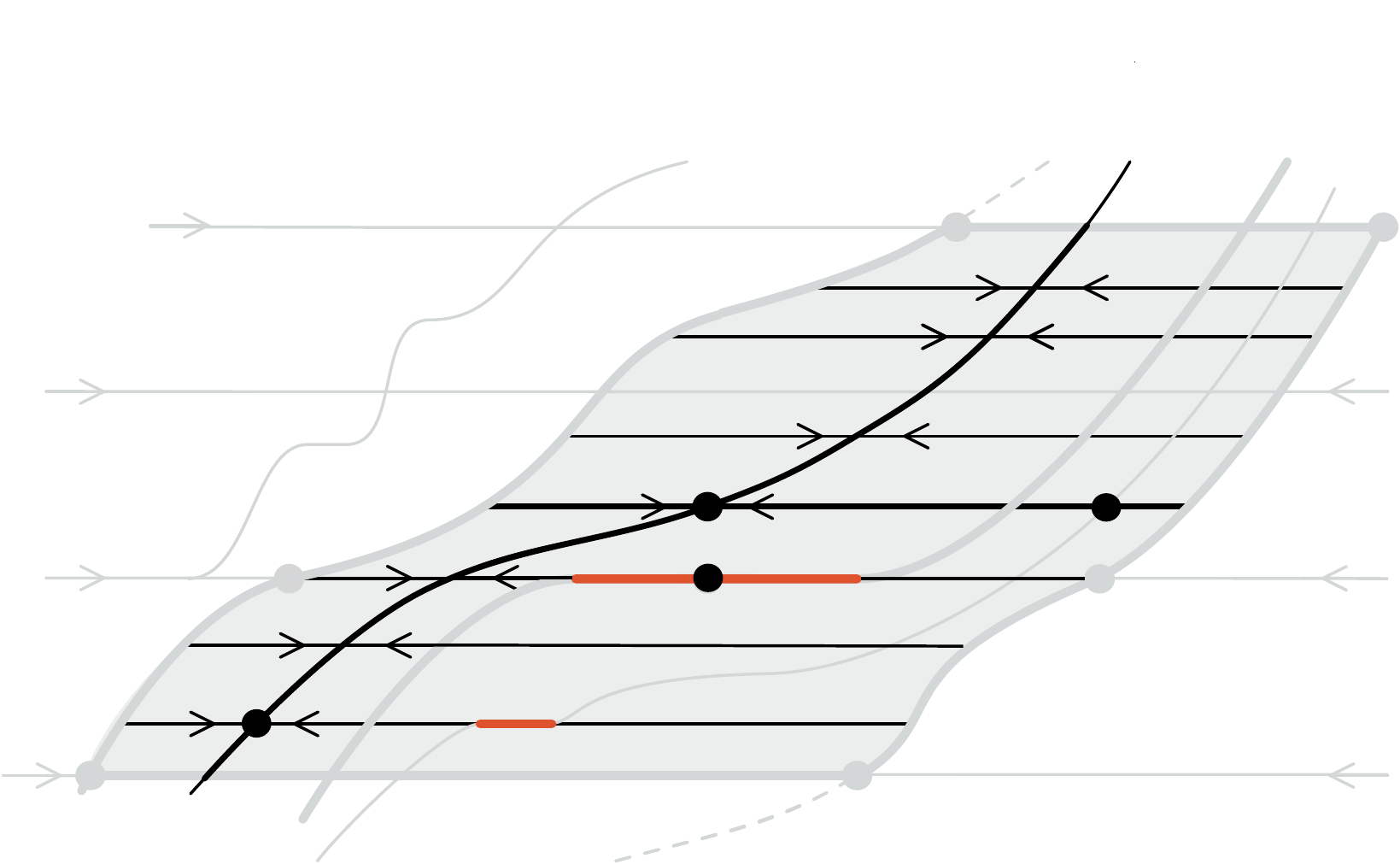}
      \put(50,27.5){\tiny$\bar\theta^\ast$}
      \put(56,21.5){\tiny$\mathcal I(\bar\theta)$}
      \put(34,11.5){\tiny$\mathcal I(\bar\eta)$}
      \put(88,25){\tiny$R^\s(r)$}
      \put(10,2){\tiny$R^\u(s)$}
      \put(-4,2){\tiny$[\bar\eta^{+},\bar\eta^{-}]$}
      \put(98,49){\tiny$[\bar\eta^{-},\bar\eta^{+}]$}
      \put(60,2){\tiny$\bar\eta^{+}$}
      \put(67,49){\tiny$\bar{\eta}^-$}
\end{overpic}
\caption{Parametrization based on a recurrent expansive point in $U$ (shaded region). All points in each strip (e.g. $\mathcal I(\bar\eta)$) have \emph{one} common pair $(r,s)$ of parameters}
\label{fig:5}
\end{figure}
Let $c^{\s}(\bar\eta^{-})$ denote the curve that is formed by the piece of $ W^\s_\Sigma(\bar\eta ^{-})$  bounded by $\bar\eta ^{-}$ and $[\bar\eta ^{-},\bar\eta ^{+}]$ and let $c^{\s}(\bar\eta ^{+})$ be the curve that is formed by $ W^\s_\Sigma(\bar\eta ^{+})$  bounded by $[\bar\eta ^{+},\bar\eta ^{-}]$ and $\bar\eta ^{+}$.
Analogously, let $c^{\u}(\bar\eta ^{-})$ denote the curve that is formed by the piece of $ W^\u_\Sigma(\bar\eta ^{-})$  bounded by $\bar\eta^{-}$ and $[\bar\eta ^{+},\bar\eta ^{-}]$ and let $c^{\u}(\bar\eta ^{+})$ be the curve that is formed by $ W^\u_\Sigma(\bar\eta ^{+})$  bounded by $\bar\eta ^{+}$ and $[\bar\eta ^{-},\bar\eta ^{+}]$.
Given $\bar\eta\in U$, the set $ W^\s_\Sigma(\bar{\eta})$ splits the closure of $U$ into two open disjoint regions, one of them containing $c^\s(\bar\eta^{-})$  and the other one containing $c^\s(\bar\eta^{+})$.
Analogously, each set $ W^\u_\Sigma(\bar{\eta})$  splits the closure of $ U$ into two open disjoint regions, one of them containing the curve $c^{\u}(\bar\eta^{-})$ and the other one the curve $c^{\u}(\bar\eta^{+})$.

For every $r\in(a,a')$ and $s\in(b,b')$ the point $[R^{\u}(r),R^{\s}(s)]$ is contained in $U$.
Indeed, each curve $ W^\s_\Sigma(R^\u(r))$ contains points of both  $c^{\u}(\bar\eta^{-})$ and $c^{\u}(\bar\eta^{+})$, so for every $s \in (b,b')$ by the Jordan Curve Theorem and the continuity of the stable foliation it has to cross $ W^\u_\Sigma(R^\s(s))$.
 Since by Lemma \ref{boundarycurve} each intersection of the form $W^\s_\Sigma(\bar{\eta}) \cap  W^\u_\Sigma(\bar{\eta})$ for $\bar{\eta} \in  U$  is a class in a strip that is contained in $U$, we get that $[R^\u(r),R^\s(s)] \in U$.

Moreover, the parameterizations $R^\s,R^\u$ induce continuous parameterizations of their quotients
\[
	\bar{\chi}\circ R^\s \colon (a,a') \longrightarrow \bar{\chi}(U)
	,\quad
	\bar{\chi}\circ R^\u \colon (b,b') \longrightarrow \bar{\chi}(U),
\]
because strips of orbits in the center stable and center unstable sets of $\bar\theta^\ast$ are trivial.
Therefore, the map
\[\begin{split}
	 &h\colon (a,a')\times (b,b') \longrightarrow \bar{\chi}(U), \\
	  &h(r,s) \eqdef
	   \bar{\chi}([R^\u(r),R^\s(s)])
	   =  [(\bar{\chi}\circ R^\u)(r),(\bar{\chi}\circ R^\s)(s)]
\end{split}\]	
defines a homeomorphism from an open rectangle onto $\bar{\chi}(U)$.
Indeed, the map $h$ is already a bijection restricted to the dense subset of intersections between the
center stable and center unstable sets of dense orbits intersecting $U$.
Taking the quotient we have that the map $h$ is continuous and injective in its image because
each set of the form $\bar{\chi}([R^\u(r),R^\s(s)])$ is just a point in the quotient.
By construction, the inverse of $h$ is also continuous. So we have a homeomorphism of a rectangle
onto  $\bar{\chi}(U)$. By the Brower's Open Mapping Theorem,
the image of $h$ is an open 2-dimensional subset, which shows the claim.
\end{proof}

To conclude the proof of the proposition we now apply the quotient flow to the section $\bar{\chi}(U)$ which is a topological surface. Indeed, by construction, each set of the form
$$
	\bigcup_{\lvert t \rvert <\tau}
	\tilde{\psi}_{t}\big(\bar{\chi}(U)\big)
$$
is continuously bijective to the product $(a,a') \times (b,b') \times (-\tau,\tau)$. By the Claim  this correspondence
is a homeomorphism and therefore $\bar{\chi}(U)$ is an open 3-dimensional set. This finishes the proof of the proposition.
\end{proof}

\begin{proof}[Proof of Theorem \ref{Hausdorff}.]
Proposition \ref{productstructure} implies that  each point in $\bar{X}$  has an open set that is continuously parametrized by an open subset of ${\mathbb R}^{3}$
which characterizes a topological 3-manifold. Hence, by \cite{kn:Bing,kn:Moise} the space
$\bar{X}$ has a smooth structure which is compatible with the quotient topology. Since the quotient
$X$ is locally homeomorphic to $\bar{X}$, the above assertions extend to $X$.
\end{proof}

\begin{remark}\label{rem:manifolds}
The fact that the quotient space $X$ is a smooth compact manifold is very important in many respects. Any smooth manifold admits a Riemannian metric and
hence there exists a distance $d\colon X \times X \longrightarrow {\mathbb R}$ which endows $X$ with a structure of a complete metric space. By the definitions of $\bar{X}$ and $X$, it is straightforward to see that the map
$$
	\bar{\Pi}\colon \bar{X} \longrightarrow X,\quad
	\bar{\Pi}(\bar{\chi }(\bar\theta)) = \chi (\bar\pi(\bar\theta))
$$ is a covering map, where $\bar\pi\colon T_{1}\tilde{M} \longrightarrow T_{1}M$ is the natural projection and $\chi\colon T_1M\to X$ and $\bar\chi\colon T_1\tilde M\to\bar X$ are the quotient maps inducing the quotient spaces.
The pullback $\bar{d}$ of $d$ to $\bar{X}$ by $\bar{\Pi}$ provides a structure of a complete metric space $(\bar{X},\bar{d})$ locally isometric to $(X,d)$. The metric $d$ is continuous and hence the quotient flow is continuous with respect to $d$.
We observe that the projection $\bar\Pi$ of the basis $\bar\chi(A_{(\cdot)}(\cdot))$ defined in Lemma~\ref{basis} naturally defines a basis in $X$.

Let ${\rm Isom}(\bar{X})$ be the group of isometries of $(\bar{X},\bar{d})$, which contains a representation $\Gamma$ of the fundamental group $\pi_{1}(X)$.
Notice that for every covering isometry $\beta$  in $T_{1}\tilde{M}$ the composition $\bar{\chi} \circ \beta$ induces a deck transformation $\bar{\beta} \colon \bar{X} \longrightarrow \bar{X}$ that satisfies $\bar{\chi} \circ \beta = \bar{\beta} \circ \bar{\chi}$ and is an element of $\Gamma$.
\end{remark}

\section{The dynamics of the quotient flow}\label{sec:dynquoflo}

We start this section by recalling some general definitions.
Given a general complete continuous flow $\psi_{t}\colon X \longrightarrow X$ acting on a complete metric space $(X,d)$, the \emph{strong stable set} $W^{\ss}(x)$ of a point $x \in X$ is the set of points $y \in X$ such that
$$
	\lim_{t \rightarrow +\infty}d(\psi_{t}(y),\psi_{t}(x)) =0.
$$
The \emph{strong unstable set} of a point $x \in X$ is defined to be the strong stable set of $x$ with respect to $\psi_{-t}$ and denoted by $W^{\uu}(x)$. The \emph{center stable set} $W^{\cs}(x)$ of a point $x\in X$ is the set of points $y\in X$ such that
$$ d(\psi_{t}(y),\psi_{t}(x)) \leq C $$
for some $C >0$ and every $t \geq 0$. The \emph{center unstable set}  is defined to be the center stable set of $x$ with respect to $\psi_{-t}$ and denoted by $W^{\cu}(x)$. For $x\in X$ and $\varepsilon>0$ let
\[\begin{split}
	W^{\cs}_\varepsilon(x)
	&\eqdef \{y\in W^{\cs}(x)\colon 
		d(\psi_t(y),\psi_t(x))\le\varepsilon \text{ for every }t\ge0\},\\
	W^{\cu}_\varepsilon(x)
	&\eqdef \{y\in W^{\cu}(x)\colon 
		d(\psi_{-t}(y),\psi_{-t}(x))\le\varepsilon \text{ for every }t\ge0\}\,.
\end{split}\]

The flow $\psi_t$ is said to have \emph{local product structure}%
\footnote{In other references such as, for example~\cite{Bow:72,Tho:91}, they  are sometimes called \emph{canonical coordinates}, which we avoid because of its similarity to geometric objects.}
 if for each sufficiently small $\varepsilon>0$ there is $\delta>0$ such that for every $x,y\in X$ with $d(x,y)\le\delta$ there is a unique $\tau=\tau(x,y)$ with $\lvert\tau\rvert\le\varepsilon$ satisfying $W^{\cs}_\varepsilon(\psi_\tau (x))\cap W^{\cu}_\varepsilon(y)\ne\emptyset$.

\bigskip
In the remainder of this section we consider  the quotient space $(X,d)$, the quotient map $\chi\colon T_1M\to X$, and the quotient flow $\psi_t\colon X\to X$ defined in Section~\ref{quotientstructure} and we describe the dynamical properties of this flow.
The following theorem is the main result of this section.

\begin{theorem} \label{expansivity}
The quotient flow
$\psi_{t} \colon X \longrightarrow X$ has the following properties:
\begin{enumerate}
\item [(i)] The flow is expansive.
\item [(ii)] For every $\theta\in T_1M$ the center stable set (center unstable set)  of $\chi(\theta)$ is the quotient of the center stable set (center unstable set) of $\theta$ with respect to the geodesic flow.
\item [(iii)] For every $\theta\in T_1M$   the strong stable set (strong unstable set) of $\chi(\theta) \in X$ is the quotient of  ${\mathscr F}^{\s}(\theta)$ (of ${\mathscr F}^{\u}(\theta)$).
\item [(iv)] If $(p,v) \neq (q,-w)$ then the center stable set of $[(q,w)]$ intersects the center unstable set of $[(p,v)]$ at a single orbit of $\psi_{t}$. In particular, the flow has a local product structure.
\item [(v)] The flow is topologically transitive.
\item [(vi)] Each strong stable set (strong unstable set)  is dense.
\item [(vii)] The flow is topologically mixing.
\end{enumerate}
\end{theorem}

We shall prove Theorem \ref{expansivity} in several steps and complete its proof at the end of this section. In the forthcoming sections we shall give some interesting applications of Theorem \ref{expansivity} and we shall continue exploring the regularity of the quotient space $X$.

Theorem \ref{transitivity} asserts many density properties of dynamical objects associated to the geodesic flow of $(M,g)$. Since in the quotient topology an open set is a set whose pre-image under the quotient map is open in $T_{1}M$, all such properties are inherited by the quotient flow $\psi_{t}$ in a straightforward way.

\subsection{Expansiveness}

Intuitively, expansiveness is a property  we should expect since the quotient collapses strips which are the only ``obstructions'' to it. 
We will need the following basic result of the theory of metric spaces.

\begin{lemma} \label{compactexhaustion}
Let $Y$ be a smooth manifold which admits a complete metric space structure $(Y,D)$.
Suppose that there exists a sequence of compact sets $(K_{n})_{n\ge1}$ such that
\begin{enumerate}
\item [(i)] for every $n\ge1$ there exists an open neighborhood $V_{n}$ of $K_{n}$ contained in $K_{n+1}$,
\item [(ii)] $\bigcup_{n}K_{n} = Y$.
\end{enumerate}
Then for every $p \in Y$ and every sequence $(x_n)_{n\ge1}$ of points $x_{n} \in K_{n+1} \setminus K_{n}$, we have that
$D(p, x_{n}) \rightarrow +\infty$ if $n \rightarrow +\infty$.
\end{lemma}

\begin{proof}
Suppose, by contradiction, that there exists $L>0$ such that $d(p, x_{n}) \leq L$ for every $n \ge1$. Since $(Y,D)$ is complete and the closed ball $\cl B(p,L)$ of radius $L$ centered at $p$ is compact, the sequence $(x_{n})_n $ has a subsequence $(x_{n_k})_k$ converging to a point $q \in \cl B(p,L)$.
 By item (2), there exists
$m \in {\mathbb N}$ such that $q \in K_{m}$. By item (1), $q \in V_{m} \subset K_{m+1}$ where $V_{m}$ is an open set. So there is $k_0>0$ such that $x_{n_k} \in V_{m}$ for every $k \geq k_0$. This contradicts the choice of the sequence $(x_{n_k})_k$, since $x_{n_k} \in K_{n_k + 1} \setminus K_{n_k}$ for every $k$ and $V_{m} \subset K_{m+1} \subset K_k$ for every $k \geq m+1$.
\end{proof}

Now we show that $\bar{X}$ has a sequence of compact sets $K_{n}$ satisfying the assumptions of Lemma \ref{compactexhaustion}. 

\begin{lemma} \label{exhaustionX}
Given $\bar\theta\in T_{1}\tilde{M}$,  consider the family of balls
\[
	\cl B_{d_S}(\bar\theta,r) \eqdef \big\{ \bar\eta\in T_{1}\tilde{M}\colon
			d_S(\bar\theta,\bar\eta) \leq r\big\},
\]
where $d_{S}$ is the induced Sasaki distance. Let $Q>0$ be the constant provided by Proposition~\ref{newMorseshadowing}.
Then  the sequence of compact sets $(K_n)_{n\ge1}$ in the space $\bar{X}$ given by
\[
	K_n \eqdef \chi\big(\cl B_{d_S}(\bar\theta,3Qn)\big)
\]	
satisfies the hypotheses of Lemma \ref{compactexhaustion}.
\end{lemma}

\begin{proof}
By Lemma \ref{width}, the width of a  strip  is bounded from above by  $Q$. So the
width of the equivalence class of every point in $T_{1}\tilde{M}$ is bounded from above by $Q$ since  flat strips in $\tilde{M}$ are isometric to strips of bi-asymptotic orbits of the geodesic flow of $T_{1}\tilde{M}$. Given $\bar\theta\in T_1\tilde M$, let $C_{r} \subset T_{1}\tilde{M}$ be the union of all classes of points in $\cl B_{d_S}(\bar\theta,r)$.
By the triangular inequality, $C_{r}$ is a compact connected set whose diameter is at most  $r +Q$. Thus, since the interior of $C_{r+3Q}$ contains the open ball of radius $r+3Q$ centered at $\bar\theta$, the set $C_{r+3Q}$ contains $C_{r}$.
Moreover, there exists a cover ${\mathcal U}_{r}$ of $C_{r}$ by open sets taken from the family  constructed in  Lemma~\ref{basis} which is contained in the interior of $C_{r+3Q}$. Indeed, we can cover any equivalence class with a set in this family which is arbitrarily close to the class.

Consider the sets
$$
	K_{n} \eqdef \chi(C_{3Qn})= \chi(\cl B_{d_S}(\bar\theta,3Qn))
$$
for $n \in {\mathbb N}$. Each such set is compact because it is the continuous image of a compact set. Clearly, $\bigcup_nC_{3Qn}$ covers all of  $T_{1}\tilde{M}$ and hence $\bigcup_nK_{n}$ covers $\bar{X}$.
By the choice of the radius $3Qn$, we have that $C_{3Qn} \subset C_{3(n+1)Q}$
for every $ n \in {\mathbb N}$, and hence $K_{n} \subset K_{n+1}$ for every integer $n$. This yields item (ii) of Lemma~\ref{compactexhaustion}.
Moreover,  by the previous paragraph the open neighborhood $\chi({\mathcal U}_{3nQ}) $ of $\chi(C_{3Qn}) = K_{n}$
is in the interior of $K_{n+1}$ for each integer $n$. This implies item (i) of Lemma~\ref{compactexhaustion}.
\end{proof}	

\begin{lemma} \label{expansive-1}
	The flow  $\bar{\psi}_{t}$ is expansive.
\end{lemma}

\begin{proof}
By contradiction, suppose that that the flow is not expansive, that is, that there exist two points $[\bar{\theta}],[\bar{\eta}]\in \bar{X}$ whose orbits have  bounded Hausdorff distance, say, bounded by $L>0$.
Up to some reparameterization we can assume that $\bar{d}([ \bar\theta], [\bar\eta]) \leq L$ and that there is some increasing homeomorphism $\rho\colon\bR\to\bR$ satisfying $\rho(0)=0$ such that for every $t\in\mathbb R$ we have
$$
	\bar{d}\big(\bar{\psi}_{t}([\bar\theta]),
			\bar{\psi}_{\rho(t)}([\bar\eta])\big)
	\leq L\,.
$$
We need the following intermediate result.

\begin{claim}
	There exists $\bar{L}>0$ such that for every $t\in\mathbb R$ we have
	\[
	 	d_{S}(\phi_{t}(\bar\theta), \phi_{\rho(t)}(\bar\eta)) \leq \bar{L},
	\]	
	where $d_{S}$ is the Sasaki distance.
	\end{claim}

\begin{proof}
Fix some fundamental domain $\cD$ of $T_1\tilde M$ containing $\bar\theta$, and let $D$ denote its diameter. Given $t\in\bR$, let $\beta_{t} \colon T_{1}\tilde{M} \longrightarrow T_{1}\tilde{M}$ be a covering isometry such that $\beta_{t}(\phi_{t}(\bar{\theta})) \in \cD$.
Hence, for every $t\in\bR$ we have
\[
	 d_{S}(\bar\theta, \beta_{t}(\phi_{t}(\bar{\theta}))) \leq D.
\]	 

As in Lemma~\ref{exhaustionX} and its proof, let $C_r \subset T_{1}\tilde{M}$ be the union of all classes of points in $\cl B_{d_S}(\bar\theta,r)$ and let $K_n:=\chi(C_{3Qn})$.

By contradiction, suppose that there exists some sequence $(t_n)_n$ such that the infimum of the Sasaki distance $d_{S}$ between $\phi_{t_{n}}(\bar\theta)$ and $\{\phi_{\tau}(\bar\eta)\colon \tau \in {\mathbb R}\}$,
is attained at $\tau=\tau_n$ satisfying
$$
	d_{S}(\phi_{t_{n}}(\bar\theta), \phi_{\tau_n}(\bar\eta))
	\geq 3Qn +D+Q \,.
$$
This would imply that $\beta_{t_{n}}(\phi_{\tau}(\bar\eta)) \notin C_{3Qn}$ for every $\tau \in {\mathbb R}$ by the definition of $C_{3Qn}$. Taking the quotient we would get
\[
	(\bar\chi\circ \beta_{t_{n}})(\phi_{\tau}(\bar\eta)) \notin K_{n}
\]
for every $\tau \in {\mathbb R}$ and every $n\ge1$. 
By Lemma \ref{compactexhaustion} we would get that for each $\tau \in {\mathbb R}$
\[
	\bar d\big(\bar\chi(\bar\theta),(\bar\chi\circ \beta_{t_{n}})(\phi_{\tau}(\bar\eta))\big)\to\infty 
\] as $n\to\infty$.
But on the other hand we would have $(\bar\chi\circ \beta_{t_{n}})(\phi_{t_{n}}(\bar\theta))\in \bar\chi(\cD)$ for every $n\ge1$, which together would imply that for every $\tau \in {\mathbb R}$
$$
	\bar{d}\big( (\bar\chi\circ \beta_{t_{n}}\circ\phi_{t_{n}})(\bar\theta),
			(\bar\chi\circ \beta_{t_{n}}\circ\phi_{\tau})(\bar\eta)\big)
	\rightarrow \infty
$$
if $n \rightarrow \infty$. Recall that by Remark~\ref{rem:manifolds}, each $\beta_t$ induces a deck transformation $\bar \beta_t$ which satisfies $\bar\chi\circ\beta_t={\bar\beta}_t\circ\bar\chi$. Hence  for every $\tau \in {\mathbb R}$
\[
	\bar{d}\big( ({\bar\beta}_{t_n}\circ\bar\chi\circ\phi_{t_{n}})(\bar\theta),
			({\bar\beta}_{t_n}\circ\bar\chi\circ\phi_{\tau})(\bar\eta)\big)
	= \bar d\big( {\bar\beta}_{t_n}([\phi_{t_n}(\bar\theta)]),{\bar\beta}_{t_n}([\phi_\tau(\bar\eta)]) \big)
	\rightarrow \infty
\]
Since each ${\bar\beta}_{t_n}$ is an isometry with respect to $\bar d$,  this together with the definition~ of the quotient flow $\bar\psi$ would imply		
\[
	\sup_{\tau\in\bR}\bar d\big( [\phi_{t_n}(\bar\theta)],[\phi_\tau(\bar\eta)] \big)
	=	\sup_{\tau\in\bR}\bar d\big( \bar\psi_{t_n}([\bar\theta]),\bar\psi_\tau([\bar\eta]) \big)
	\to\infty
\]
But the latter term is bounded from above by $L$ for all $n$. This contradiction shows that such a sequence $(t_{n})_n$ does not exists. This proves the claim.
\end{proof}

Let us consider the orbits of $\bar\theta$ and $\bar\eta$ under the flow $\phi_t$
in $T_{1}\tilde{M}$.
Observe that the Claim implies that the strips $F(\bar\theta)$ and $F(\bar\eta)$ are within a distance
of $\bar{L} + \bar{Q}$, where $\bar{Q}$ depends on the constant  in Proposition \ref{newMorseshadowing}. Thus, taking the canonical projection
from $T_{1}\tilde{M}$ onto $\tilde{M}$ we get that the strips of the geodesics $\gamma_{\bar\theta}$ and $\gamma_{\bar\eta}$
are within a distance $\bar{L} + \bar{Q}$ (recall that the canonical projection is a Riemannian submersion). This can only happen when
both geodesics are bi-asymptotic and hence  are in the same strip. So $F(\bar\theta)= F(\bar\eta)$ and therefore their quotients  in $\bar{X}$ coincide.

This implies that the the quotient flow $\bar{\psi}_{t}$ is expansive.
\end{proof}

\begin{lemma}\label{expansive-2}
	The flow $\psi_t$ is expansive.
\end{lemma}

\begin{proof}
Note that by Remark~\ref{rem:manifolds} $\bar{X}$ covers $X$. Thus, by compactness of $X$, there exists $r>0$ such that the covering map restricted to any ball of radius $r$ is a homeomorphism.

Arguing by contradiction, suppose  that $\psi_t$ would not be expansive, that is,  suppose that there are  two points $[\theta],[\eta]\in X$ with distinct orbits and that there exists a homeomorphism $\rho\colon\bR\to\bR$ satisfying $\rho(0)=0$ and such that for every $t \in {\mathbb R}$ we have
$$
	d\big(\psi_{t}([\theta]), \psi_{\rho(t)}([\eta])\big) \leq r.
$$
Then we could lift these orbits to a pair of orbits in $\bar{X}$ which would stay within a distance $r$ from each other. Therefore, by Lemma~\ref{expansive-1}, these orbits would coincide. Thus, the orbits of $[\theta]$ and $[\eta]$ would coincide, which gives a contradiction.
\end{proof}

\subsection{Invariant sets and heteroclinic relation}\label{sec:invsethomocl}

We will consider the following invariant sets of the quotient flows $\bar{\psi}_{t}$ and $\psi_{t}$. Recall the definition of the projection map $\bar\chi\colon T_1\tilde M\to \bar X$ in Section~\ref{quotientstructure}.
Given $\theta \in T_{1}M$ and one of its lifts $\bar{\theta} \in T_1\tilde{M}$, let
\[\begin{split}
	&\tilde W^\ast([\bar\theta]):=\bar\chi\big(\tilde {\mathscr F}^\ast(\bar\theta)\big)\,,
	\quad
	W^\ast([\theta]):=\chi\big( {\mathscr F}^\ast(\theta)\big),\quad
	\ast=\cs,\cu,\\
	&\tilde W^{\ss}([\bar\theta]):=\bar\chi\big(\tilde {\mathscr F}^\s(\bar\theta)\big)\,,
	\quad
	W^{\ss}([\theta]):=\chi\big( {\mathscr F}^\s(\theta)\big),\quad
	\\
	&\tilde W^{\uu}([\bar\theta]):=\bar\chi\big(\tilde {\mathscr F}^\u(\bar\theta)\big)\,,
	\quad
	W^{\uu}([\theta]):=\chi\big( {\mathscr F}^\u(\theta)\big)
	\,.
\end{split}\]
In this way, $X$ is the union of the sets $W^{\cs}([\theta])$, $[\theta] \in X$, which strongly indicates that $X$ should be  foliated by the above collection of sets.
Analogously, $X$ is the union of the sets $W^{\ast}([\theta])$, $[\theta] \in X$,
for $\ast=\ss,\cu,\uu$ respectively.
In this section we will study the dynamical properties of these sets and will establish that indeed each set $W^{\ss}([\theta])$ is a strong stable set of the quotient flow as defined above,  justifying our notation.

\begin{lemma} \label{centraldistance}
Given $D>0$ there exists $D'>0$ such that for every $\bar{\theta} \in T_{1}\tilde{M}$, and every pair of points $[\bar{\tau}], [\bar{\eta}]\in \tilde {\mathscr F}^\s(\bar\theta)$
such that $\bar{d}([\bar{\tau}], [\bar{\eta}]) \leq D$, we have that
$$
	\bar{d}\big(\bar{\psi}_{t}([\bar{\tau}]), \bar{\psi}_{t}([\bar{\eta}]) \big) \leq D'
$$
for every $t \geq 0$.
\end{lemma}

\begin{proof}
Given $\bar\theta\in T_1\bar M$, let  $\bar{\tau},\bar{\eta}\in{\tilde {\mathscr  F}}^{\s}(\bar{\theta})$ as in the hypothesis and note that the distance between their orbits is non-increasing since we consider a surface $(M,g)$ without focal points. 

\begin{claim}
	If $K \subset \bar{X}$ is compact then $\bar{\chi}^{-1}(K)\subset T_{1}\tilde{M}$ is compact.
\end{claim}

\begin{proof}

We know that $\bar{\chi}^{-1}(K)$ is closed by continuity of $\bar{\chi}$.
By Lemma \ref{basis} the sets $\bar{\chi}(A)$ for $A=A_{\bar{\theta}}(\tau,\varepsilon, \delta,\bar{\theta}^{-},\bar{\theta}^{+},\bar\eta^-,\bar\eta^+)$ 
as in this lemma  form a basis for the quotient topology. Here each set of the basis satisfies
\[
 	\bar{\chi}^{-1}\big(\bar{\chi}(A) \big)	
	= A.
\]	
Each set $A$ is an open subset with compact closure. Since $K$ is compact, it can be covered by a finite
collection $\{\bar{\chi}(A_{i})\}_{i=1,\ldots,m}$ of sets in this basis. By Lemma \ref{basis}, this implies that $\bar{\chi}^{-1}(K)$ is covered by the finite collection $A_{i}$ of open
sets. Since the union of the closures of the sets $A_{i}$ is a compact set, $\bar{\chi}^{-1}(K)$ is a closed subset of a compact set. Thus $\bar{\chi}^{-1}(K)$ is compact.
\end{proof}

Given $D>0$, let $L=L(D)>0$ be such that for every pair of points $\bar\theta_1, \bar\theta_2 \in T_{1}\tilde{M}$
with $\bar{d}( [\bar\theta_1],[\bar\theta_2]) \leq D$ we have $d_{S}(\bar\theta_1,\bar\theta_2) \leq L$. Indeed, such  constant $L$ exists by the above Claim, since the preimage by
$\bar{\chi}$ of the closed ball of radius $D$ centered at $[\bar\theta_1]$ is a compact subset of $T_{1}\tilde{M}$. 

So for $D := \bar{d}([\bar{\tau}], [\bar{\eta}])$
we have $d_{S}(\bar{\tau}, \bar{\eta}) \leq L$ and hence
$$ 
	d_{S}(\phi_{t}(\bar{\tau}), \phi_{t}(\bar{\eta})) \leq L
$$
for every $ t \geq 0$. By co-compactness of $T_{1}\tilde{M}$, given a fundamental domain $\cD$, there exist representatives $\beta_{t}(\phi_{t}(\bar{\tau})), \beta_{t}(\phi_{t}(\bar{\eta}))$ of $\bar{\tau}, \bar{\eta}$ by isometries $\beta_{t}$ of the fundamental group of $T_{1}M$ such that
\begin{itemize}
\item $\beta_{t}(\phi_{t}(\bar{\tau})) \in \cD$ for each $t >0$,\\[-0.4cm]
\item $d_{S}\big(\bar{\tau}, \bar{\eta}\big) = d_{S}\big(\beta_{t}(\phi_{t}(\bar{\tau})), \beta_{t}(\phi_{t}(\bar{\eta}))\big)$ for every $t >0$.
\end{itemize}

Since the image by $\bar{\chi}$ of a compact set is compact, there exists a constant $D'= D'(D)$ such that
$$
	\bar{d}\big( [\beta_{t}(\phi_{t}(\bar{\tau}))],[ \beta_{t}(\phi_{t}(\bar{\eta}))]\big) \leq D' .
$$
But by definition, for each covering isometry $\beta$ acting on $T_{1}\tilde{M}$ we have an induced isometry $\bar{\beta}$ of $\bar{d}$
which acts as
$$
	\bar{\beta}([\bar{\theta}])
	= (\bar{\beta}\circ \bar{\chi})(\bar{\eta}) =( \bar{\chi} \circ \beta)(\bar{\eta}) =  [\beta(\bar{\theta})].
$$
So we get
$$
	\bar{d}\big( [\beta_{t}(\phi_{t}(\bar{\tau}))],[ \beta_{t}(\phi_{t}(\bar{\eta}))]\big)
	= \bar{d}\big(\bar{\psi}_{t}([\bar{\tau}]), \bar{\psi}_{t}([\bar{\eta}])\big)
	\leq D'
$$
for every $t \geq 0$,  concluding the proof.
\end{proof}

\begin{lemma} 
The quotient flow $\bar{\psi}_{t}$ is uniformly contracting on the stable sets $\tilde{W}^{\ss}([\bar{\theta}])$ in the following sense:
for every $D,\varepsilon >0$, there exists $t_0=t_0(D, \varepsilon) >0$ such that for every $\bar{\theta} \in T_{1}\tilde{M}$ and every two points $[\bar{\eta}], [\bar{\xi}] \in \tilde{W}^{\ss}([\bar{\theta}])$ satisfying $\bar{d}([\bar{\eta}], [\bar{\xi}] ) \leq D$
we have
$$ \bar{d}\big(\bar{\psi}_{t}([\bar{\xi}]), \bar{\psi}_{t}([\bar{\eta}])\big) \leq \varepsilon$$
for every $t \geq t_0$.
\end{lemma}

\begin{proof}
The proof follows from a standard argument of expansive dynamics which we only sketch.
By contradiction, suppose that there exist numbers $D,\varepsilon>0$ and sequences $([\bar{\eta}_{n}])_n, ([\bar{\xi}_{n}])_n\subset \tilde{W}^{\ss}([\bar{\theta}_{n}])$ such that for every $n\ge1$
\begin{equation}\label{secondprop}
	\bar{d}\big([\bar{\eta}_{n}], [\bar{\xi}_{n}] \big) \leq D
	\quad\text{ and }\quad
	\bar{d}\big(\bar{\psi}_{n}([\bar{\xi}_{n}]), \bar{\psi}_{n}([\bar{\eta}_{n}])\big) \geq  \varepsilon\,.
\end{equation}
By Lemma \ref{centraldistance} there exists $D'>0$ such that
 $\bar{d}(\bar{\psi}_{t}([\bar{\eta}_n]),\bar{\psi}_{t}([\bar{\xi}_n])) \leq D'$
for every $ t \geq 0$.

By co-compactness, there is a fundamental domain $\cD$ and we can find a subsequence $(n_k)_k$ and covering isometries $\beta_{n_k}\colon T_1\tilde M\to T_1\tilde M$ such that for every $k\ge1$ we have
\[
	\hat\eta_{n_k}\eqdef \beta_{n_k}(\bar\eta_{n_k})\in \cD,\quad
	\hat\xi_{n_k}\eqdef \beta_{n_k}(\bar\xi_{n_k})\in \cD,
\] 
and such that these subsequences converge
\[
	\hat\eta_{n_k}\to \bar\eta_\infty,\quad
	\hat\xi_{n_k}\to\bar\xi_\infty
\]
and that
\[
	\bar d(\bar\psi_t(\hat\eta_{n_k}),\bar\psi_t(\hat\xi_{n_k}))\le D'
\]
for every $t\ge -n_k$. It follows from our assumption that $\bar\eta_\infty$ and $\bar\xi_\infty$ are
distinct points whose orbits are
bi-asymptotic. On the other hand, by Lemma \ref{expansive-2} the flow $\bar{\psi}_{t}$ is
expansive, so these two orbits
coincide. But this contradicts the second property of the sequences $ [\bar{\xi}_{n}],[\bar{\eta}_{n}] $ in~\eqref{secondprop}.
This proves the lemma.
\end{proof}

Another important property of the invariant sets of the quotient flows is the following.

\begin{lemma}
For each $\ast=\cs,\ss,\cu,\uu$, the set $\tilde{W}^{\ast}([\bar{\theta}])$ varies continuously in $[\bar{\theta}] \in \bar{X}$, uniformly on compact sets, with respect to the Hausdorff topology.
\end{lemma}

\begin{proof}
Recall that, by
 Theorem~\ref{central set}, each family of invariant sets constitutes a continuous foliation by $C^{1}$ leaves. So they vary continuously on compact sets with respect to the Hausdorff topology. In fact, in $T_{1}\tilde{M}$ the invariant foliations are continuous with respect to the $C^{1}$ compact open  topology, that is stronger than continuity with respect to the Hausdorff topology. The quotient preserves continuity properties, so the same holds for the quotient invariant sets.
\end{proof}

\subsection{Further  properties}

\begin{lemma} \label{heteroclinicquotient}
Given $\bar\theta=(p,v)\in T_1\tilde M$, for every $\bar\eta=(q,w)\in T_1\tilde M$ such that $(q,-w)\ne(p,v)$ the set $\tilde{W}^{\cs}([\bar\eta]) \cap \tilde{W}^{\cu}([\bar{\theta}])$ consists of a single orbit of $\bar{\psi}_{t}$.
\end{lemma}

\begin{proof}
By Theorem~\ref{transitivity} item (iii), the intersection ${\tilde {\mathscr  F}}^{\cs}(\bar\eta) \cap {\tilde {\mathscr  F}}^{\cu}(\bar{\theta})$ is nonempty and consists of a strip of bi-asymptotic orbits. The quotient map preserves this intersection. Observing that each strip becomes a single orbit in the quotient space, we deduce the statement.
\end{proof}

\begin{proof}[Proof of Theorem~\ref{expansivity}]
		Lemma \ref{expansive-2} shows item (i).
	The results in Subsection~\ref{sec:invsethomocl} imply items (ii) and (iii).
	Lemma~\ref{heteroclinicquotient} proves item (iv).

What remains to show are items (v)--(vi). We will sketch their proof using the basis of sets $\bar\chi(A_{(\cdot)}(\cdot))$ constructed in Lemma~\ref{basis} and their projection to $X$. As observed in the end of Section~\ref{quotientstructure}, the projection $\bar\Pi\colon\bar X\to X$ defines a local homeomorphism and hence the basis $\bar\chi(A_{(\cdot)}(\cdot))$ naturally projects to a basis in $X$.

Item (v) claims the transitivity of the quotient flow. This follows from the transitivity of the geodesic flow of $(M,g)$ (Theorem \ref{transitivity} item (ii)) and the fact that the family of sets $ A_{\bar\theta}(\cdot)$ established in Lemma \ref{basis}  is a family of open neighborhoods of
sections of strips $\mathcal I(\bar\theta)\subset T_1\tilde M$. Indeed, each dense orbit must intersect any open set and hence each dense orbit intersects any set $A$ from the basis specified in Lemma \ref{basis}.
Hence, the quotient of a dense orbit must intersect any open set $\bar\chi(A)$. This proves the density of the quotient of any dense orbit in the quotient topology.

Item (vi) claims the minimality of the quotient of each strong stable set ${W}^{\ss}([\theta])$ and each strong unstable set   ${W}^{\uu}([\theta])$.
By Theorem \ref{transitivity} item (i), the horocycle foliations of $T_{1}M$ are minimal. The same argument applied in the previous paragraph shows that the quotient of any dense set of $T_{1}M$ is dense in the quotient $X$.

Item (vii) claims that the quotient flow $\psi_{t}$ is topologically mixing. This follows from Theorem \ref{transitivity} item (ii)
and the fact that the sets from the basis specified in Lemma \ref{basis} form a family of open neighborhoods in $T_{1}\tilde M$. Hence, the mixing property is verified in particular in these sets. Taking the quotient, this property is verified as well by the quotient flow.

This finishes the proof.
\end{proof}

\begin{proof}[Proof of Theorem~\ref{quotient}]
The theorem is a consequence of Proposition~\ref{productstructure} and Theorem~\ref{expansivity}.
\end{proof}

\section{Measures of maximal entropy}\label{sec:entropystatement}

In this setting we study the entropy of a flow. We first establish some general results and finally prove Theorem~\ref{maximalentropy}.

\subsection{Expansive flows with local product structure}

In this subsection we start by considering a general continuous flow $\psi_t\colon X\to X$ without singular points on a compact metric space $X$. Of course, we have in mind the quotient flow defined in Section~\ref{quotientstructure}.

Given positive numbers $a$ and $\delta$ we call a pair of sequences $(x_k)_{k=k_0}^{k_1}$ of points $x_k\in X$ and numbers $(\tau_k)_{k=k_0}^{k_1}$ ($k_0=-\infty$ or $k_1=\infty$ are permitted), a \emph{$(\delta,a)$-pseudo orbit} for the flow $(\psi_t)_t$ if for every $k=k_0,\ldots,k_1-1$ we have $\tau_k\ge a$ and $d(\psi_{\tau_k}(x_k),x_{k+1})<\delta$.
Denote $s_0=0,s_k=\tau_0+\ldots+\tau_{k-1}$, and $s_{-k}=\tau_{-k}+\ldots+\tau_{-1}$.
A $(\delta,a)$-pseudo orbit $(x_k,\tau_k)_k$ is \emph{$\varepsilon$-traced} by a true orbit $(\psi_t(y))_{t\in \bR}$ if there is some increasing homeomorphism $\alpha\colon\bR\longrightarrow\bR$ satisfying $\alpha(0)=0$, such that for every $k=0,1,\ldots$ for every $t\ge0$ satisfying $s_k\le t<s_{k+1}$ we have
\[
	d(\psi_{\alpha(t)}(y),\psi_{t-s_k}(x_k))\le\varepsilon
\]
and for every $k=1,2,\ldots$ for every $t\le0$ satisfying $-s_{-k}\le t\le -s_{-k+1}$
\[
	d(\psi_{\alpha(t)}(y),\psi_{t+s_{-k}}(x_{-k}))\le\varepsilon
\]
The flow $(\psi_t)_t$ is said to have the \emph{pseudo orbit tracing property with respect to time $a>0$} if for every $\varepsilon>0$ there is $\delta_0>0$ such that every $\delta\in(0,\delta_0)$ for every $(\delta,a)$-pseudo orbit is $\varepsilon$-traced by a true orbit of $(\psi_t)_t$. For $a=1$ we simply speak of the \emph{pseudo orbit tracing property}.

Recall the definition of local product structure in Section~\ref{sec:dynquoflo}.

\begin{proposition}[{\cite[Theorem 7.1]{Tho:91}}]\label{pro:prooop1}
	Every continuous expansive flow without singular points on a compact metric space which has local product structure has the pseudo orbit tracing property.
\end{proposition}

We say that the flow $(\psi_t)_t$ has the \emph{periodic orbit specification property}\footnote{We follow the definition in~\cite{Oka:97} which corrects some results on the uniqueness of an equilibrium state (e.g. measure of maximal entropy) in~\cite{Fra:77} under stronger hypotheses. Observe that condition 1) is slightly stronger than in usual definitions (for example in~\cite{Fra:77}). It requires that for the existence of the periodic point  in any intermediate construction step of solely considering the subset of points $\{x_0,\ldots,x_{k+1}\}\subset\{x_0,\ldots,x_n\}$ the numbers $r_k$ are, in fact, defined recursively.
} if for every $\varepsilon>0$ there is a positive number $T=T(\varepsilon)$ such that for any  integer $n\ge2$, any collection of points $x_0,\ldots,x_n\in X$, and any sequence of real numbers $t_0< t_1<\cdots<t_{n+1}$ satisfying $t_{k+1}-t_k\ge T$ for every $k=0,\ldots,n$, there is a sequence of numbers $r_0,\ldots,r_{n+1}$ having the following properties: $r_0=0$ and for every $k=0,\ldots,n$
\begin{itemize}
\item[1)] $r_{k+1}$ is determined by $x_0,\ldots,x_{k+1}$ and by $t_0,\ldots,t_{k+1}$,
\item[2)] $\lvert r_{k+1}-r_k\rvert<\varepsilon$,
\item[3)] there is a periodic point $y\in X$  with period $\tau$ satisfying
	\[
		\lvert\tau-(t_{k+1}-t_0)\rvert\le (k+1)\varepsilon
	\]	
	and satisfying for every $\ell=0,\ldots,k$ and for every $t\in [t_\ell,t_{\ell+1}-T]$
	\[
	d(\psi_{t+r_\ell}(y),\psi_{t-t_\ell}(x_\ell))<\varepsilon\,.
	\]
\end{itemize}

\begin{proposition}\label{pro:camposs}
	Every continuous expansive topologically mixing flow without singular points on a compact metric space  having the pseudo orbit tracing property has the periodic orbit specification property.
\end{proposition}

We will need the following auxiliary result (see~\cite[Proposition 3.2]{Tho:91}).

\begin{lemma}\label{lem:maisumdeTho}
	Let $(\psi_t)_t$ be a continuous expansive flow without singular points on a compact metric space which has local product structure.
	For every $\varepsilon>0$, there exists $\delta=\delta(\varepsilon)>0$ such that for every $x,y\in X$, every interval $[T_1,T_2]$ containing $0$, and for every increasing homeomorphism $\alpha\colon \bR\to\bR$ satisfying $\alpha(0)=0$ and $d(\psi_{\alpha(t)}(x),\psi_t(y))\le\delta$ for every $t\in[T_1,T_2]$ we have
\[	
	\lvert\alpha(t)-t\rvert<\varepsilon\,.
\]
\end{lemma}

\begin{proof}[Proof of Proposition~\ref{pro:camposs}]
We follow a standard idea of proof (see e.g.~\cite{KwiOpr:12}).
	Let $\varepsilon>0$.	
	Let $\varepsilon_0>0$ be an expansivity constant.
Let $\varepsilon_2\eqdef\min\{\varepsilon,\varepsilon_0/2\}$.
Let $\delta(\varepsilon)$ be as in Lemma~\ref{lem:maisumdeTho}.
Let $\delta_0=\delta_0(\min\{\delta(\varepsilon),\varepsilon_0/2\})$ be the number given by the pseudo orbit tracing property. Let $\delta\in(0,\min\{\delta_0,\varepsilon_2/2\})$.  	
	
	Take a covering of $X$ by finitely many open sets $U_i$ satisfying $\diam U_i\le \delta$. Since the flow is topologically mixing, there exists $T>0$ such that for every index pair $i,j$ we have
	\begin{equation}\label{zezes}
		\psi_t(U_i)\cap U_j\ne\emptyset\quad
		\text{ for every }\lvert t\rvert\ge T/2\,.
	\end{equation}	
	
	Let $x_0,\ldots,x_n\in X$ be a sequence of points and $t_0<t_1<\ldots<t_{n+1}$ a sequence of numbers satisfying $t_{k+1}-t_k\ge T$ for every $k=0,\ldots,n$. For every $k=0,\ldots,n$ let
\[
	y_k:=\psi_{t_{k+1}-t_k-T/2}(x_k)\,.
\]
For every $k=0,\ldots,n$ there are indices $j_k$ and $i_k$ such that
\[
	x_k\in U_{j_k},\quad y_k\in U_{i_k}\,.
\]	
It follows from~\eqref{zezes} that  for every $k=0,\ldots,n-1$ there is $z_k\in U_{i_k}\subset B(y_k,\delta)$ such that $\psi_{T/2}(z_k)\in  U_{j_{k+1}}$ and that for every $k=0,\ldots,n$ there is $w_k\in U_{i_k}$ such that $\psi_{T/2}(w_k)\in U_{j_0}$.
	
	By these choices, for every $k=1,\ldots,n$ consider the pair of finite sequences of points $(\hat x_0,\ldots,\hat x_{2k+3})\eqdef(x_0,z_0,x_1,z_1,\ldots x_{k-1},z_{k-1},x_k,w_k,x_0)$ and of numbers
\[
	(\hat\tau_1,\ldots,\hat\tau_{2k+2})
	\eqdef
	(t_1-t_0-\frac T2,\frac T2,t_2-t_1-\frac T2,\frac T2,
	\ldots,t_{k+1}-t_k-\frac T2,\frac T2)
\]
that we concatenate infinitely many times and thus define a periodic $(\delta,T/2)$-pseudo orbit.
By our assumption, this pseudo orbit is $\delta(\varepsilon)$-traced by a true orbit $(\psi_t(y))_t$, that is, there is a point $y$ and some increasing homeomorphism $\alpha\colon\bR\longrightarrow\bR$ satisfying $\alpha(0)=0$ such that for every $k$ and for every $t$ satisfying $s_k\le t<s_{k+1}$
\[
	d(\psi_{\alpha(t)}(y),\psi_{t}(\hat x_k))\le \min\{\delta(\varepsilon),\varepsilon_0/2\}\,.
\]
Since the flow is expansive, by the choice of $\varepsilon_0$ this bi-infinite tracing orbit $(\psi_t(y))_t$ is uniquely determined.
As the pseudo orbit is periodic, the tracing orbit must be closed, that is, we have $\psi_\tau(y)=y$ for some $\tau>0$.
By Lemma~\ref{lem:maisumdeTho} and the choice of $\delta(\varepsilon)$, as the shadowing orbit $(\psi_{\alpha(t)}(y))_t$ is close to pieces of orbits $(\psi_t(x_k))_t$, the period $\tau$ of the tracing periodic orbit must satisfy property 3) in the definition of the periodic orbit specification property.
This finishes the proof.
\end{proof}

Propositions~\ref{pro:prooop1} and~\ref{pro:camposs} together imply the following.

\begin{corollary}\label{k:corol1}
		If a continuous expansive topologically mixing flow without singular points on a compact metric space  has local product structure then the flow has the periodic orbit specification property.
\end{corollary}

\subsection{Entropy}

A Borel probability measure is said to be \emph{invariant} under the flow $\psi_t\colon X\to X$ if it is $\psi_t$-invariant for every $t\in\bR$. Let $\cM$ denote the set of all flow-invariant Borel probability measures.
A set $Z\subset X$ is \emph{invariant} under the flow if $\psi_t(Z)=Z$ for every $t\in\bR$. A measure $\nu\in\cM$ is said to be \emph{ergodic}  if for every invariant set $Z\subset X$ we have either $\nu(Z)=0$ or $\nu(Z)=1$.

Given $\nu\in \cM$, we denote by $h_\nu(\psi_t)$ the \emph{metric entropy} with respect to the time-$t$ map $\psi_t$ (see~\cite{Wal:81} for its definition). By Abramov's formula~\cite{Abr:59}, we have
\[
	h_\nu(\psi_t)=\lvert t\rvert\,h_\nu(\psi_1)\quad\text{ for every }t.
\]	
One calls $h_\nu(\psi):=h_\nu(\psi_1)$ the \emph{metric entropy} of $\nu$ with respect to the flow $\psi=(\psi_t)_t$.

Given $\varepsilon>0$,  $T>0$, and $x\in X$, define
\[
	B(x,\varepsilon,T)\eqdef
	\{y\in X\colon d(\psi_s(y),\psi_s(x))< \varepsilon \text{ for every }s\in[0,T]\}\,.
\]	
Two points $x,y$ are called \emph{$(T,\varepsilon)$-separated}  if $y\notin B(x,\varepsilon,T)$. A set $E\subset X$ is \emph{$(T,\varepsilon)$-separated}  if every pair of distinct elements in $E$ has this property.
Given an invariant compact set  $Z\subset X$, denote by $M(T,\varepsilon,Z,\psi)$  the maximal cardinality of any $(T,\varepsilon)$-separated subset of $Z$.
The \emph{topological entropy} of $Z$ with respect to the flow $\psi$ is defined by
\begin{equation}\label{def:entropia}
	h(\psi,Z)\eqdef
	\lim_{\varepsilon\to0}\limsup_{T\to\infty}\frac{1}{T} \log
	M(T,\varepsilon,Z,\psi).
\end{equation}
We simply write $h(\psi)=h(\psi,X)$ and call this number the \emph{topological entropy} of the flow $\psi$.
 One can verify Abramov's formula also in the case of topological entropy
\[
	h(\psi_1,Z)=\frac{1}{\lvert t\rvert}h(\psi_t,Z) \quad\text{ for every }t\,.
\]
(see~\cite{Ito:69} or~\cite[Proposition 21]{Bow:71}) and, together with the variational principle (see~\cite{Wal:81}) denoting by $\cM(\psi_1,Z)$ the set of all $\psi_1$-invariant measures $\nu$ for which $\nu(Z)=1$, one has
\begin{equation}\label{varprinc}
	h(\psi,Z)= \sup_{\nu\in\cM(\psi_1,Z)}h_\nu(\psi_1)
	=h(\psi_1,Z)\,.
\end{equation}

In the following we will only work with the entropy of the time-1 map $\psi_1$.  A set of points $E$ is called \emph{$(n,\varepsilon)$-spanning} with respect to $\psi_1$ if $x, y\in E, x\ne y$, implies $d(\psi_k(x),\psi_k(y))>\varepsilon$ for some $k\in\{0,\ldots,n-1\}$.
Given $Z\subset X$ compact, let $N(n,\varepsilon,Z,\psi)$ be the maximal cardinality of a $(n,\varepsilon)$-spanning set $\{x_i\}_i\subset Z$. Recall that
\begin{equation}\label{eq:fisio}
	h(\psi_1,Z)= \lim_{\varepsilon\to0}\limsup_{n\to\infty}\frac{1}{n} \log
	N(n,\varepsilon,Z,\psi)	\,.
\end{equation}

Consider a continuous flow $\phi_{t}\colon Y\longrightarrow Y$ acting on a compact metric space $Y$ being time-preserving semi-conjugate to  the continuous flow $\psi_{t}\colon  X\longrightarrow X$ via $\chi\colon Y\to X$ by $\chi\circ\phi_t=\psi_t\circ\chi$ ($\psi_t$ is also called a \emph{factor} of $\phi_t$ and $\phi_t$ is called an \emph{extension} of $\psi_t$, of course we have in mind the quotient map introduced in Section~\ref{quotientstructure}). For $y\in Y$ consider the equivalence class
\[
	[y]\eqdef \{z\in Y\colon \chi(z)=\chi(y)\} = \chi^{-1}(\chi(y)).
\]
Note that $[y]$ is compact for every $y$. Recall that, given an $\phi$-invariant compact set $A\subset Y$ (and hence $\psi$-invariant compact set $\chi(A)\subset X$), by the common property of a factor and by~\cite[Theorem 17]{Bow:71} we have
\begin{equation}\label{eq:entsemconj}
	h(\psi_1,\chi(A))
	\le h(\phi_1,A)
	\le h(\psi_1,\chi(A)) + \sup_{x\in\chi(A)}h(\phi_1,\chi^{-1}(x)).
\end{equation}

\subsection{Maximal entropy measures}

An ergodic measure $\nu$ is a \emph{measure of maximal entropy} if it realizes the supremum in~\eqref{varprinc}.
Notice that such a measure always exists provided the entropy map $\nu\mapsto h_\nu(\psi_1)$ is upper semi-continuous~\cite{Wal:81}. This is the case if the flow is smooth~\cite{New:89} or if the map $\psi_1$ is $h$-expansive~\cite{Wal:81}.
For any axiom A flow the measure of maximal entropy is unique~\cite{Bow:73c}. In general, the continuous flows we want to study are not hyperbolic, so we have to rely on a more general result due to Franco~\cite{Fra:77} (see also Oka~\cite{Oka:97}).

In what follows let $\psi_t\colon X\to X$ be a continuous expansive flow on a compact metric space $X$ satisfying the periodic orbit specification property. The measure of maximal entropy  can be constructed explicitly from the distribution of periodic orbits as follows~\cite{Bow:72}. Let $\varepsilon>0$ be an expansivity constant. For every $T$ there is only a finite number of periodic orbits for $\psi$ with minimal period between $T-\varepsilon$ and $T+\varepsilon$, denote this family by $\Per(\psi,T-\varepsilon,T+\varepsilon)$. 
Denote by $\Per(\psi,T)$ the family of periodic orbits with minimal period less than or equal to $T$.
Denote by $\card$ the cardinality. Consider
\[	\hat\nu_{\psi,T}:=
	\frac{\sum_\gamma\nu_{\psi,\gamma}}
			{\card\Per(\psi,T-\varepsilon,T+\varepsilon)}
\]
with summation taken over all $\gamma\in\Per(\psi,T-\varepsilon,T+\varepsilon)$ and $\nu_{\psi,\gamma}$ denoting the  invariant probability measure supported on $\gamma$.
Note that
\begin{equation}\label{def:perorbmeanumbers}
	\card \Per(\psi,T-\varepsilon,T+\varepsilon) \le M(T,\varepsilon,X,\psi)
\end{equation}
(compare the proof of~\cite[Theorem 5]{kn:BW72}). We also consider
\[
	\nu_{\psi,T}:=
	\frac{\sum_\gamma\nu_{\psi,\gamma}}
			{\card\Per(\psi,T)}
\]
with summation taken over all $\gamma\in\Per(\psi,T)$.

\begin{proposition}[\cite{Fra:77}]\label{pp:Franco}
	Let $\psi_t\colon X\to X$ be a continuous expansive flow on a compact metric space satisfying the periodic orbit specification property.
	
	Then $\hat\nu_{\psi,T}$ and $\nu_{\psi,T}$ both converge in the weak$\ast$ topology to the unique measure of maximal entropy $\nu_\psi$ as $T\to\infty$. In particular, for $\varepsilon>0$ sufficiently small, we have
\[
	\lim_{T\to\infty}\frac1T\log\card\Per(\psi,T-\varepsilon,T+\varepsilon)
	=\lim_{T\to\infty}\frac1T\log\card\Per(\psi,T)
	=h(\psi)
	>0.
\]
\end{proposition}

Proposition~\ref{pp:Franco} and Corollary~\ref{k:corol1} hence imply the following.

\begin{corollary}\label{cor:uniquemmaxent}
		If a  continuous expansive  topologically mixing  flow without singular points  on a compact metric space  has local product structure then this flow has  a unique (hence ergodic) measure of maximal entropy and $\hat\nu_{\psi,T}$ and $\nu_{\psi,T}$ both converge in the weak$\ast$ topology to this measure as $T\to\infty$.
\end{corollary}

Consider a continuous flow $\phi_{t}\colon Y\longrightarrow Y$ acting on a compact metric space $Y$ and being time-preserving semi-conjugate to  the continuous flow $\psi_{t}\colon  X\longrightarrow X$ via a semi-conjugacy $\chi\colon Y\to X$. If $x$ is a periodic point for $\psi$ with period $\ell$ and $\gamma=\{\psi_t(x)\colon0\le t< \ell\}$ is its orbit then $[\beta]=\chi^{-1}(\gamma)$ is compact and $\phi$-invariant  and has period $\ell$. Hence we can pick a probability measure $\mu_{\phi,[\beta]}$ supported on $[\beta]$ which is invariant with respect to the flow  $\phi=(\phi_t)_t$.
We define
\[
	\mu_{\phi,T}\eqdef
	\frac{\sum_{[\beta]}\mu_{\phi,[\beta]}}{\card \Per(\phi,T)},
\]
with summation taken over all $[\beta]\in\Per(\phi,T)$.
The probability measure $\mu_{\phi,T}$ is $\phi$-invariant and, in particular, $\phi_1$-invariant.

We can now formulate~\cite[Theorem 1.5]{BuzFisSamVas:12} in our setting.

\begin{theorem}\label{the:Buzzietal}
	Let $\psi_t\colon X\to X$ be a continuous expansive flow on a compact metric space satisfying the periodic orbit specification property and denote by $\nu_\psi$ its unique (hence ergodic) Borel probability measure of maximal entropy. Let $\phi_{t}\colon Y\longrightarrow Y$ be a continuous flow time-preserving semi-conjugate to $\psi_t$  through some continuous surjective map $\chi\colon Y \longrightarrow X$ and assume that the following conditions are satisfied:
	\begin{enumerate}
		\item [(i)] $h(\phi_1,[y])=0$ for every $y\in Y$,\\[-0.4cm]
		\item [(ii)] $\nu_\psi(\{\chi(y)\colon [y]=\{y\}\})=1$.
	\end{enumerate}
Then $\mu_{\phi,T}$ converges in the weak$\ast$ topology as $T\to\infty$  to the unique ergodic Borel probability measure of maximal entropy for $\phi=(\phi_t)_t$.
\end{theorem}

\subsection{Lyapunov exponents}

We continue to assume that $(M,g)$ is a compact surface without focal points.
Lyapunov exponents and associated vector bundles provided by the Oseledec multiplicative ergodic theorem are in this setting conveniently described by orthogonal stable and unstable Jacobi fields (see~\cite{GalHulLaf:93} for more details). 
A Jacobi field $J$ along a geodesic $\gamma_\theta$, $\theta\in TM$, with $\gamma_\theta'(0)=\theta$ is uniquely determined by its initial conditions $J(0),J'(0)\in T_pM$.  
The set of \emph{orthogonal Jacobi fields} is the one such that $J(0)$ and $J'(0)$ are orthogonal to $v$ and are exactly those such that for every $t\in \bR$, $J(t)$ is normal to $\gamma_\theta'(t)$.
Given a vector $\xi\in T_\theta T_1M$, we denote by $J_\xi$ the Jacobi field such that $J_\xi(0)$ and $J_\xi'(0)$ are the horizontal and the vertical component of $\xi$, respectively. Recall that with this notation the linearized geodesic flow acts in the way that  for every $t\in\bR$, $J_\xi(t)$ and $J_\xi'(t)$ are the horizontal and the vertical component of $D\phi_t(\xi)$, respectively. 
An (orthogonal) Jacobi field $J$ is called \emph{stable} (respectively, \emph{unstable}) if $\lVert J(t)\rVert$ is uniformly bounded for all $t\ge0$ (respectively, uniformly bounded for all $t\le0$). We denote by $J^\s$ (respectively, $J^\u$) the set of stable (respectively, unstable) orthogonal Jacobi fields and consider the subspaces 
\[
	E^\s_\theta\eqdef \{\xi\in T_\theta T_1M\colon J_\xi\in J^\s\},\quad
	E^\u_\theta\eqdef \{\xi\in T_\theta T_1M\colon J_\xi\in J^\u\}.
\]
In our setting each such subspace is one-dimensional. 
The distributions $E^\s\colon \theta\in T_1M\mapsto E^\s_\theta\in T_\theta T_1M$ and  $E^\u\colon \theta\in T_1M\mapsto E^\u_\theta\in T_\theta T_1M$ are invariant under the linearized geodesic flow and continuous. 
The bundles $E^\s_\theta$ and $E^\u_\theta$ may not be  linearly independent. They generate the space of orthogonal Jacobi fields along a geodesic if, and only if, the curvature does not vanish everywhere along this geodesic.  In fact, the geodesic flow is an Anosov flow precisely if, and only if,  these bundles are linearly independent at \emph{every} point of the unit tangent bundle ~\cite{Ebe:73}.

We denote by $\cR$ the subset of vectors $\theta\in T_1M$ such that  $E^\s_\theta$ and $E^\u_\theta$ are linearly independent, this set is open and invariant under $(\phi_t)_t$ (a vector in $\cR$ is also said to have \emph{rank $1$}). The complementary set $\cH\subset T_1M$ is closed and invariant. Based on the quotient map $\chi$ we introduce in Section~\ref{quotientstructure}, for the corresponding equivalence classes  we have  
\begin{equation}\label{RR0inclusion}
\cR\subset\cR_0\eqdef \big\{\theta\in T_1M\colon [\theta]=\{\theta\}\big\}.
\end{equation}

Let $\lambda(\theta)$ denote the \emph{Lyapunov exponent} at $\theta$ associated to the linearization of the geodesic flow and its action on the  subbundle $E^\u$
\[
	\lambda(\theta)
	\eqdef \lim_{t\to\pm\infty}\frac1t\log\,\lVert D\phi_t|_{E^\u_\theta}\rVert,
\]
whenever both limits exist and are equal, where $\lVert\cdot\rVert$ denotes the norm induced by the Sasaki metric on $TM$. The Lyapunov exponent $\lambda(\theta)$ is well defined for any so-called \emph{Lyapunov regular} vector $\theta$. By the above facts, in our setting, $\lambda(\theta)$ can in fact be computed as the Birkhoff average of a continuous function which is the instantaneous rate of expansion caused by the derivative of the geodesic flow acting on the subspace $E^\u_\theta$. It vanishes on $\cH$ because unstable Jacobi fields are covariantly constant along geodesics tangent to vectors in $\cH$. 

The set of Lyapunov regular vectors is of full measure with respect to any invariant Borel probability measure. It coincides at almost every point with the nonnegative Lyapunov exponent provided by the Oseledec decomposition.
Ruelle's inequality~\cite{Rue:78} asserts that for all $\phi$-invariant probability measures $\mu$ we have
\begin{equation}\label{eq:ruelllle}
	h_\mu(\phi_1)\le \lambda(\mu)\eqdef \int_{T_1M}\lambda(\theta)\,d\mu(\theta).
\end{equation}

Let $\widetilde m$ be the Liouville measure $m$ restricted to $\cR$ and normalized to obtain a probability measure. Ergodicity of $\widetilde m$ was proved in~\cite{kn:Pesin}. In all known examples $\widetilde m$ coincides with the Liouville measure, but this has not been proved in general. Since $\widetilde m$ is absolutely continuous and has positive density throughout $\cR$ we have $h_{\widetilde m}(\phi_1)=\lambda(\widetilde m)>0$, by Pesin's formula~\cite{Pes:77b}. Thus, the variational principle for topological entropy~\eqref{varprinc} implies 
\[
	h(\phi_1)
	\ge h_{\widetilde m}(\phi_1)
	=\lambda(\widetilde m)>0.
\]
The following is an immediate consequence of Ruelle's inequality~\eqref{eq:ruelllle} and the variational principle for entropy \eqref{varprinc} (recalling that $\cH$ is closed and invariant).

\begin{lemma}\label{lem:zeroentropy}
	We have $h_\mu(\phi_1)=0$ for every invariant probability measure $\mu$ satisfying $\mu(\cH)=1$. We have $h(\phi_1,\cH)=0$.
\end{lemma}	

\subsection{Small entropy on strips}

We now return to our setting in the rest of the paper. In this section let $\phi_t\colon T_1M\to T_1M$ be the geodesic flow of a  $C^{\infty}$ compact connected boundaryless surface $(M,g)$ without focal points and of genus greater than one.
By Proposition~\ref{prop:semiconjugation}, $\phi=(\phi_t)_t$ is  time-preserving semi-conjugate to the quotient flow $\psi=(\psi_t)_t$ from Definition~\ref{def:quotientflow} through the quotient map $\chi\colon T_1M\to X$. 
We will need the following lemma.

\begin{lemma}\label{lem:lembowlocl}
	$h(\phi_1,[\theta])=0$ for any $\theta\in T_1M$.
\end{lemma}

\begin{proof}
	The result is trivial for every $\theta$ for which $[\theta]=\{\theta\}$.
	
Assume that $[\theta]$ contains a point $\eta=(q,w)\ne\theta$. Then $\chi(\eta)=\chi(\theta)$ and by Definition~\ref{def:quotientflow} we have  $\eta\in {\mathscr  F}^{\s}(\theta)$. If $\tilde{\mathscr  F}^{\ss}(\bar\theta)$ is any lift of ${\mathscr  F}^{\s}(\theta)$ in $T_1\tilde M$ and $\bar\eta\in\tilde {\mathscr  F}^{\ss}(\bar\theta)$ a lift of $\eta$ then the geodesics $\gamma_{\bar\theta},\gamma_{\bar\eta}$ are bi-asymptotic and by Theorem~\ref{flatstrip} bound a nontrivial strip in $\tilde M$ and stay in constant distance for every $t \in {\mathbb R}$.
Hence, for every $\delta>0$ a $(1,\delta)$-spanning set of $[\theta]$ is $(n,\delta)$-spanning for every $n$. Hence, in particular $N(n,\delta,[\theta],\phi)\le N(1,\delta,[\theta],\phi)$ and thus $h(\phi_1,[\theta])=0$ by~\eqref{eq:fisio}.
\end{proof}

Now Lemma~\ref{lem:lembowlocl} and~\eqref{eq:entsemconj} together imply the following.
\begin{lemma}\label{lem:projection}
	For any compact invariant set $Z\subset T_1M$  we have $h(\phi_1,Z)=h(\psi_1,\chi(Z))$.
\end{lemma}

By Theorem \ref{expansivity} and Corollary \ref{cor:uniquemmaxent}, the flow $\psi=(\psi_t)_t$ has a unique  (hence ergodic) measure $\nu_\psi$ of maximal (positive) entropy. 

\begin{lemma}\label{lem:bigmeasure}
	We have $\nu_\psi(\{\chi(\theta)\colon[\theta]=\{\theta\}\})=1$.
\end{lemma}

\begin{proof}
	With~\eqref{RR0inclusion} we have $\{\chi(\theta)\colon[\theta]=\{\theta\}\}=\chi(\cR_0)$. Note that $\cR_0$ and its complementary set $\cH\setminus\cR_0$ are $\phi$-invariant and hence $\chi(\cR_0)$ and $\chi(\cH\setminus\cR_0)$ are $\psi$-invariant.
	
	By contradiction, suppose that $\nu_\psi(\chi(\cR_0))<1$. Hence, by ergodicity, $\nu_\psi(\chi(\cH\setminus\cR_0))=1$ which implies $\nu_\psi(\chi(\cH))=1$. Since $\cH$ is compact and invariant, by the variational principle~\eqref{varprinc} we then have $h(\psi_1,\chi(\cH))>0$.
Hence, by Lemma~\ref{lem:projection} we obtain $h(\phi_1,\cH)>0$. This is however in contradiction with Lemma~\ref{lem:zeroentropy}.
\end{proof}

We are now ready to give the proof of Theorem~\ref{maximalentropy}.

\begin{proof}[Proof of Theorem~\ref{maximalentropy}]
Apply Theorem~\ref{the:Buzzietal} using Lemmas~\ref{lem:lembowlocl} and ~\ref{lem:bigmeasure}.
\end{proof}

For completeness, we formulate the following result. For compact rank one surfaces (in fact, for rank one manifolds of any dimension) this is contained in~\cite[Corollary 6.2]{Kni:98}. Based on the semi-conjugacy, using the above introduced notation for $T>0$ the set $\Per_\cR(T):=\Per(\phi,T)\cap\cR$ is the set of primitive periodic orbits $\gamma$ in $\cR$ of period $\ell(\gamma)\le T$. 

\begin{corollary}\label{corollary611}
	Let $\phi_t\colon T_1M\to T_1M$ be the geodesic flow of a  $C^{\infty}$ compact connected boundaryless surface $(M,g)$ without focal points and of genus greater than one.
We have $0=h(\phi_1,\cH)<h(\phi_1)$ and
\[
	\lim_{T\to\infty}\frac 1T\log \#\Per_{\cR}(T)=h(\phi_1)\,.
\]
\end{corollary}

\begin{proof}
By Lemma~\ref{lem:zeroentropy} we have $h(\phi_1,\cH)=0$.
By Proposition~\ref{pp:Franco} we have
\[
	\lim_{T\to\infty}\log\#\Per(\psi,T)=h(\psi)>0.
\]
Hence, by Lemma~\ref{lem:projection} applied to $T_1M$ and~\eqref{varprinc} we have $h(\phi_1)=h(\psi_1)=h(\psi)>0$.

Since $h(\psi,\chi(\cH))=0$ by the above  and by Lemma~\ref{lem:projection}, by definition~\eqref{def:entropia} of the entropy on $\chi(\cH)$, for any $\delta\in(0,h(\psi))$ there exists $\varepsilon>0$ sufficiently small (and smaller than an expansivity constant for $\psi$) and $T_0=T(\varepsilon)$ sufficiently large, such that for every $T\ge T_0$ we have
\[
	M(T,\varepsilon,\chi(\cH),\psi)
	\le e^{T\delta}.
\] 
By~\eqref{def:perorbmeanumbers}, we conclude
\[
	\Per(\psi,T-\varepsilon,T+\varepsilon)\cap \chi(\cH)
	\le M(T,\varepsilon,\chi(\cH),\psi)
	\le e^{T\delta}
\]Since by Proposition~\ref{pp:Franco} we have
\[
	\lim_{T\to\infty}\frac1T\log\#\Per(\psi,T-\varepsilon,T+\varepsilon)
	= h(\psi)>0,
\]
the above implies
\[
	\lim_{T\to\infty}\frac1T\log\#\Per(\psi,T-\varepsilon,T+\varepsilon)\cap\chi(\cR)
	= h(\psi)>0
\]
from which we conclude
\[
	\lim_{T\to\infty}\frac1T\log\#\Per(\psi,T)\cap\chi(\cR)
	= h(\psi).
\]
Since $\chi$ is a homeomorphism on $\cR$ which preserves parametrization, this and the above immediately imply
\[
	\lim_{T\to\infty}\frac1T\log\#\Per_{\cR}(T)=h(\psi)= h(\phi_1).
\]
This proves the corollary.
\end{proof}

\section*{Notations}

\begin{equation*}
  \begin{aligned}
	M \hspace{0.1cm}
&	\overset{\pi}{\longleftarrow}
&	\sigma_t^{\bar\theta}\colon \tilde M \to \tilde M,\,\, H^\pm(\bar\eta) \hspace{0.8cm}
&
\\[0.5cm]
 	\theta=(p,v)\in T_1M	 \hspace{1.3cm}			
&	
&	\bar\theta=(p,v)\in T_1\tilde M 	\hspace{1.3cm}	
&\\
	\phi_t\colon T_1M \to T_1M, \,\,\mathscr F^\ast(\theta) 		
&	\overset{\bar\pi}{\longleftarrow}	
&	\phi_t\colon T_1\tilde M \to T_1\tilde M,\,\,\tilde{\mathscr F}^\ast(\bar\theta)
&	\quad \ast=\s,\u,\cs,\cu\\
       	\downarrow \chi \hspace{2.9cm}	
&
&	\downarrow \bar\chi \hspace{2.9cm}	
&
\\	
        \psi_t\colon X \to X,\,\,W^\ast([\theta]) 	\hspace{0.6cm}		
&	\overset{\bar\Pi}{\longleftarrow}	
&	\bar\psi_t\colon \bar X\to \bar X,\,\,\tilde W^\ast([\bar\theta]) \hspace{0.6cm}
&\quad \ast=\ss,\uu,\cs,\cu
  \end{aligned}
\end{equation*}

\end{document}